\documentclass[a4paper,11pt]{article}

\usepackage[mathscr]{euscript}

\usepackage{amsfonts,amscd,amssymb,amsmath,amsthm,bbm}
\usepackage{centernot,caption,subcaption,color,url,tikz}
\usepackage[T1]{fontenc}
\usepackage{authblk}
\usepackage{graphicx}
\usepackage{caption}
\usepackage{subcaption}
\usepackage[colorlinks=true,linkcolor=blue,citecolor=red]{hyperref}
\usepackage{hyperref}
\usepackage[capitalize]{cleveref}

\usepackage[margin=1in]{geometry}
\usepackage{setspace}

\newtheorem{theorem}{Theorem}
\newtheorem{thm}{Theorem}[section]

\newtheorem{lemma}[thm]{Lemma}
\newtheorem{claim}[thm]{Claim}

\theoremstyle{remark}
\newtheorem{remark}{Remark}[section]

\newcommand{\1}{\mathbbm{1}}

\usepackage{amsmath,amssymb,amsfonts,latexsym,epsfig,graphicx,longtable,wrapfig}

\title{Non-isomorphic subgraphs in random graphs}

\author{\Large Michael Krivelevich\thanks{School of Mathematical Sciences, Tel Aviv University, Tel Aviv 6997801, Israel.  Research supported in part by NSF-BSF grant 2023688.
\newline
Email: krivelev@tauex.tau.ac.il. 
}, \,\,\,\,\, Maksim Zhukovskii\thanks{The University of Sheffield, School of Computer Science, Sheffield S1 4DP, UK.\newline Email: m.zhukovskii@sheffield.ac.uk.}}

\date{}

\begin{document}

\maketitle

\begin{abstract}
We establish the asymptotic behaviour of $\mu(G(n,p))$, the number of unlabelled induced subgraphs in the binomial random graph $G(n,p)$, for almost the entire range of the probability parameter $p=p(n)\in[0,1]$. In particular, we show that typically the number of subgraphs becomes exponential when $p$ passes $1/n$, reaches maximum possible base of exponent (asymptotically) when $p\gg 1/n$, and reaches the asymptotic value $2^n$ when $p$ passes $2\ln n/n$. For $p\gg \ln n/n$, we get the first order term and asymptotics of the second order term of $\mu(G(n,p))$. We also prove that random regular graphs $G_{n,d}$ typically have $\mu(G_{n,d})\geq 2^{c_d n}$ for all $d\geq 3$ and some positive constant $c_d$ such that $c_d\to 1$ as $d\to\infty$.
\end{abstract}

\section{Introduction}
\label{sc:Intro}

For a graph $G$ let $\mu(G)$ be the number of non-isomorphic induced subgraphs in $G$ (or, in other words, the number of  different {\it unlabelled} induced subgraphs in $G$). This parameter naturally evaluates ``diversity'' of a graph. In particular, for a clique or an empty graph $G$, $\mu(G)=|V(G)|+1$, which is smallest possible. It is natural to expect that for graphs that do not have large homogeneous sets, this diversity parameter is much larger. Indeed, it was conjectured by Erd\H{o}s and R\'{e}nyi that any {\it $c$-Ramsey graph} $G$ on $n$ vertices (i.e., a graph in which all homogeneous induced subgraphs are of size at most $c\log_2 n$) have exponentially  many non-isomorphic induced subgraphs: $\mu(G)>2^{\varepsilon n}$ for some constant $\varepsilon=\varepsilon(c)>0$\footnote{Recall that any graph on $n$ vertices has either a clique or an independent set of size at least $\frac{1}{2}\log_2 n$, so $c\geq\frac{1}{2}$.}. The conjecture was resolved by Shelah in~\cite{Shelah}. Even though the relation between $\mu(G)$ and the size of the largest homogeneous induced subgraph in $G$ has been investigated (see, e.g.,~\cite{AB89,AH91,EH89}), for non-Ramsey graphs fairly tight bounds on $\mu(G)$ are not known. In particular, it is easy to see that there exist graphs $G$ with linear in $|V(G)|$ homogeneous sets whose $\mu(G)$ is still exponential in $|V(G)|$ --- this is the case for a disjoint union of an $n$-clique and a $c$-Ramsey graph on $n$ vertices. Another, perhaps less artificial, example is an $n$-comb graph $G$ --- that is, a tree obtained by joining all vertices of an $n$-path with an independent set of size $n$ via a matching of size $n$ --- it has $\mu(G)\geq 2^{n-3}=2^{|V(G)|/2-3}$.

The asymptotic behaviour of $\mu(G)$ for a binomial random graph $G\sim G(n,p)$ played a key role in the resolution of the reconstruction conjecture for random graphs by Bollob\'{a}s~\cite{Bol-rec} and by M\"{u}ller~\cite{Muller}. It was also exploited by Bonnet, Duron, Sylvester, Zamaraev, and the second author~\cite{BDSZZ} to prove, for every $C>0$, the existence of tiny monotone classes of graphs that do not admit adjacency labeling schemes with labels of size at most $C\log n$. Let us recall that whp\footnote{With high probability, that is, with probability tending to 1 as $n\to\infty$.} the binomial random graph $G(n,p)$ with constant probability of appearance of an edge $p\in(0,1)$ is $c$-Ramsey~\cite{BE:cliques}, for some $c=c(p)>0$. Therefore, whp $G(n,p)$ has $2^{\Theta(n)}$ non-isomorphic subgraphs. Actually, a stronger result holds: In 1976~\cite{Muller} M\"{u}ller proved that whp $\mu(G(n,1/2))=2^{(1-o(1))n}$. 
  Recently~\cite{BDSZZ}, Bonnet, Duron, Sylvester, Zamaraev, and the second author estimated a threshold for the property of having exponentially many non-isomorphic induced subgraphs:
\begin{itemize}
\item if, for some constant $\varepsilon>0$, $p\leq\frac{1-\varepsilon}{n}$, then whp $\mu(G(n,p))=2^{o(n)}$, and 
\item if, for a large enough $C>0$, $\frac{C}{n}\leq p\leq\frac{1}{2}$, then whp $\mu(G(n,p))=2^{\Theta(n)}$. 
\end{itemize}
Note that, when $p=o(1)$, $G(n,p)$ is not Ramsey whp --- say, when $pn>C$ for large enough $C>0$, its independence number is concentrated around $\frac{2\log(np)}{p}$ whp~\cite{Frieze}. In particular, when $p=C/n$, the independence number is of order $\Theta(n)$ whp. Nevertheless, $\mu(G(n,p))$ is still exponential whp for large enough $C>0$. The reason for the subexponential bound when $p\leq\frac{1-\varepsilon}{n}$ is that, in this regime, whp $G(n,p)$ is a disjoint union of trees and unicyclic graphs of size $O(\log n)$. Due to the 
foundational discovery of Erd\H{o}s and R\'{e}nyi~\cite{ER-evol}, the structure of $G(n,p)$ changes dramatically when $p$ crosses a tiny interval around $1/n$: if $p\geq(1+\varepsilon)/n$, then with high probability $G(n,p)$ contains a unique giant component of linear size. It is then natural to expect that the giant component is diverse enough to guarantee an exponential lower bound on $\mu(G(n,p))$ whp. In this paper, among other results, we show that this is indeed the case --- $\frac{1}{n}$ is a sharp threshold for the property of containing  exponentially many non-isomorphic induced subgraphs, improving the result from~\cite{BDSZZ}.

Since the edge-complement of $G(n,p)$ is distributed as $G(n,1-p)$ and since a graph and its edge-complement have exactly the same number of non-isomorphic induced subgraphs, we get that, for $G\sim G(n,p)$ and $G'\sim G(n,1-p)$, the random variables $\mu(G)$ and $\mu(G')$ are identically distributed. Therefore, it suffices to consider the case $p\leq 1/2$ and everywhere below we assume this restriction.

In this paper, we describe asymptotics of $\mu(G(n,p))$ for almost the entire range of probabilities $p=p(n)\in[0,1]$. In particular, we get the asymptotics of 
\begin{itemize}
\item the second order term of $\mu(G(n,p))$ when $p\gg \frac{\ln n}{n}$, 
\item $\mu(G(n,p))$ when $p\geq (2+\varepsilon)\frac{\ln n}{n}$, and 
\item $\log\mu(G(n,p))$ when $p\gg\frac{1}{n}$. 
\end{itemize}
Here is our main result.

\begin{theorem}
Let $p:=p(n)\in[0,1/2]$, $G\sim G(n,p)$, and let $\varepsilon>0$.
\begin{enumerate}
\item If $\frac{\ln n}{n}\ll p\leq\frac{1}{2}$, then whp
\begin{equation}
\mu(G)=2^n-2^{n(1-2p(1-p))+\sqrt{8np(1-p)(1-2p(1-p))\ln n}(1-o(1))}.
\label{eq:mu_very_dense}
\end{equation}
\item If $(2+\varepsilon)\frac{\ln n}{n}\leq p\leq\frac{1}{2}$, then whp $\mu(G)=(1-o(1))2^n$.
\item If $p\leq(2-\varepsilon)\frac{\ln n}{n}$, then whp $\mu(G)=o(2^n)$.
\item There exists $C=C(\varepsilon)>0$ such that, for any $p\in[C/n,1/2]$, whp $\mu(G)\geq 2^{(1-\varepsilon)n}$.
\item Let $\varepsilon$ be sufficiently small. If $np\geq 1+\varepsilon$, then whp $\mu(G)\geq 2^{\varepsilon n/1000}$. If $np\leq 1+\varepsilon$, then whp $\mu(G)\leq 2^{3\varepsilon n}$.
\item There exist $0<c_1=c_1(\varepsilon)\leq c_2=c_2(\varepsilon)<1$ such that the following holds. If $np\leq 1-\varepsilon$, then whp $\mu(G)\leq 2^{n^{c_2}}$. If $np\geq 1-\varepsilon$, then whp $\mu(G)\geq 2^{n^{c_1}}$. Moreover, $c_1,c_2$ are decreasing continuous functions of $\varepsilon\in(0,1)$ such that $c_1(0+)=c_2(0+)=1$,  $c_1(1-)=c_2(1-)=0$ and $c_1,c_2=1-\Theta(\varepsilon^2)$ as $\varepsilon\to 0$.
\end{enumerate}
\label{th:bin_dense}
\end{theorem}

Theorem~\ref{th:bin_dense} establishes the following three transition points: typically the number of subgraphs becomes exponential when $p$ passes $1/n$, reaches maximum possible base of exponent (asymptotically) when $p\gg\frac{1}{n}$, and reaches the asymptotic value $2^n$ when $p$ passes $\frac{2\ln n}{n}$.  It also estimates the order of magnitude of $\log\mu(G(n,p))$ for $p$ around $1/n$: 
$$
\log_2\mu(G(n,(1+\varepsilon)/n))=\Theta(\varepsilon)n\, \text{ and } \, \log_2\mu(G(n,(1-\varepsilon)/n))=n^{1-\Theta(\varepsilon^2)}\quad\text{whp.}
$$
Perhaps the most interesting missing regimes are $p=\frac{1\pm o(1)}{n}$ and $p=\frac{(2\pm o(1))\ln n}{n}$.

Below in this section, we describe the proof strategy of Theorem~\ref{th:bin_dense}. One of the novel ingredients in our proof is the estimation of the expected number of non-isomorphic rooted subtrees of a Galton--Watson tree, which may be of independent interest. We prove that, for a Galton--Watson tree with offspring distribution $\mathrm{Pois}(1-\varepsilon)$, the expected number of non-isomorphic rooted subtrees equals $\exp(\Omega(1/\varepsilon))$. The respective claim is presented in Section~\ref{sc:trees}.

We also show that the number of non-isomorphic induced subgraphs of a random $d$-regular graph $G_{n,d}$ (see, e.g.,~\cite[Chapter 9]{Janson}) inherits the behaviour of this statistics for the supercritical binomial random graph $G(n,p=C/n)$, that is, the  following analogue of the fourth part of Theorem~\ref{th:bin_dense} holds in random regular graphs.

\begin{theorem}
For every $\varepsilon>0$ there exists a large enough $d_0$ such that, for every integer $d\geq d_0$, a  uniformly random $d$-regular graph $G\sim G_{n,d}$ on $[n]$ satisfies $\mu(G)\geq 2^{(1-\varepsilon)n}$ whp. 
\label{th:regular}
\end{theorem}

This result is essentially tight --- the constant factor in front of $n$ in the power of the exponent has to be bounded away from 1, see details in Section~\ref{sc:regular-tight}. The proof of Theorem~\ref{th:regular} is presented in Section~\ref{sc:regular-proof}. It is interesting to see whether the randomness in Theorem~\ref{th:regular} is essential and its statement can be generalised to pseudorandom graphs. Nevertheless, this discussion falls outside the scope of the present work and is left for future consideration. Finally, we show that $\mu(G_{n,d})$ is exponential whp for every $d\geq 3$.

\begin{theorem}
\label{th:expanders}
Let $d\geq 3$ be a fixed integer constant. Then there exists $c=c(d)>0$ such that whp $\mu(G)\geq 2^{cn}$. 
\end{theorem}

The proof of Theorem~\ref{th:expanders} is presented in Section~\ref{sc:regular-proof2}.


\paragraph{Strategy and structure of the proof of Theorem~\ref{th:bin_dense}.} In the case $p\gg\frac{\ln n}{n}$, we show in Section~\ref{sc:proof_1} that, for almost all induced subgraphs $H$ in $G\sim G(n,p)$ that have another isomorphic induced subgraph, there exist two vertices $x,x'$ such that $V(H)\setminus\{x\}$ is entirely inside the union of common neighbourhood and common non-neighbourhood of $x$ and $x'$. The second-order term in~\eqref{eq:mu_very_dense} follows: whp the maximum number of vertices that have identical adjacencies to $x,x'$ equals $n(1-2p(1-p))+(1-o(1))\sqrt{8np(1-p)(1-2p(1-p))\ln n}$. 

When $p\leq(2-\varepsilon)\frac{\ln n}{n}$, using the standard second-moment argument, we show in Section~\ref{sc:proof_3} that whp, for a fixed set $U\subset V(G)$ of size $(1/2+o(1))n$, there exist a vertex $x\in U$ and sufficiently many vertices $x'$ outside $U$ such that $x$ and $x'$ do not have neighbours in $U$. It means that $G[U]$ and $G[U\cup\{x'\}\setminus\{x\}]$ are isomorphic for all such $x'$. Since a typical subset of $V(G)$ has $(1/2+o(1))n$ vertices, this observation implies the third claim in Theorem~\ref{th:bin_dense}. It turns out that such isolated vertices is the main reason for having many isomorphic induced subgraphs when $p$ is around $2\ln n/n$: if $p\geq\frac{(2+\varepsilon)\ln n}{n}$, then whp the number of non-isomorphic induced subgraphs is $(1-o(1))2^{n}$. See details in Section~\ref{sc:proof_2}.

In order to prove part 4 of Theorem~\ref{th:bin_dense}, we show in Section~\ref{sc:proof_4} that, when $p>C/n$, a typical subset $U\subset V(G)$ has at most $2^{\varepsilon n}$ other subsets that induce graphs isomorphic to $G[U]$. We prove the latter in the following way: first, we show that whp all subsets of $U$ of size at least $\varepsilon^2 n$ have sufficiently many edges. Then, if a set $U'$ has at least $\varepsilon^2 n$ vertices outside of $U$, any isomorphism $G[U]\to G[U']$ has to move all the edges from $G[U\setminus U']$, which is very unlikely. On the other hand, there are less than $2^{\varepsilon n}$ sets that have at least $|U|-2^{\varepsilon^ 2n}$ common vertices with $U$, implying the required bound.

Note that, for $p\leq C/n$, whp $\log_2\mu(G)/n$ is bounded away from 1. Indeed, in this regime, it is well known whp $G$ has $(1+o(1))n(1-p)^n\geq (e^{-C}+o(1))n$ isolated vertices. 
 Therefore, whp $\log_2\mu(G)\leq (1-e^{-C}+o(1))n$. We also note that, in the strictly supercritical regime $1+\varepsilon\leq np=O(1)$, our lower bound on $\mu(G)$ gives a linear dependency on $\varepsilon$, which is asymptotically tight: if $np=1+\varepsilon$, then whp the giant component $L_1$ of $G$ has size $(2\varepsilon+O(\varepsilon^2))n$ and all the other components have size $O(\log n)$ and contain at most one cycle (see, e.g.,~\cite[Theorems 5.4, 5.10]{Janson}). It is then easy to show that the main contribution to $\mu(G)$ is due to subgraphs of the giant component (see, e.g.,~\cite[Theorem 1.4 (i)]{BDSZZ} for an analogous argument). Therefore, whp $\mu(G_n)\leq (1+o(1))2^{|V(L_1)|}=2^{(2\varepsilon+O(\varepsilon^2))n}$. 

Our proof of part 5 of Theorem~\ref{th:bin_dense} relies on an estimate of the expected number of non-isomorphic subtrees in a Galton--Watson tree as well as on a contiguous model of the giant component due to Ding, Lubetzky, and Peres~\cite{DLP}. We prove that a rooted $\mathrm{Pois}(1-\varepsilon)$--Galton--Watson tree has $\exp(\Theta(1/\varepsilon))$ non-isomorphic subtrees with the same root, on average --- see Claim~\ref{cl:GW-main} in Section~\ref{sc:trees}.

Finally, the fact that whp $\mu(G)=2^{o(n)}$ whenever $np\leq 1-\varepsilon$ is proved in~\cite{BDSZZ}. We show in Section~\ref{sc:proof_6} that the same proof strategy gives the desired bound $\mu(G)\leq 2^{n^{c_2}}$ for some $c_2=c_2(\varepsilon)\in(0,1)$. In order to prove that whp $\log_2\mu(G)\geq n^{c_1}$ for $np\geq 1-\varepsilon$ and some $c_1=c_1(\varepsilon)\in(0,1)$, we combine the following two assertions that hold for a suitable choice of $k=\Theta(\ln n)$: 1)  whp there are at least $n^{c_1}$ connected components $T_1,\ldots,T_m$ isomorphic to trees on $k$ vertices in $G$; 2) whp there are no two components isomorphic to the same tree on $k$ vertices in $G$. It immediately implies the desired bound because disjoint unions of trees from any two different subsets of $\{T_1,\ldots,T_m\}$ are not isomorphic.


\paragraph{Notation.} For a graph $\Gamma$ and a set of vertices $U\subset V(\Gamma)$, we denote by $\Gamma[U]$ the subgraph of $\Gamma$ induced by $U$. We use double braces  $\{\!\!\{\cdot\}\!\!\}$ to distinguish a multiset from a set.

\section{Preliminaries}

For $n\in\mathbb{N},$ we let 
$$
\mathcal{J}_n:=\left[n/2-\sqrt{n\ln n},n/2+\sqrt{n\ln n}\right].
$$
We let $G\sim G(n,p)$ be a random graph on the vertex set $[n]:=\{1,\ldots,n\}$.

\begin{claim}
Let $\mathbf{U}$ be a uniformly random subset of $[n]$ and let $\mathcal{F}_n$ be a family of pairs $(U,H)$, where $U\subset[n]$, and $H$ is a graph on $[n]$. Let $p=p(n)\in[0,1]$ and let $G\sim G(n,p)$ be sampled independently of $\mathbf{U}$. If there exists $\varphi(n)$ such that, for every $m\in\mathcal{J}_n$ and every $U\in{[n]\choose m}$, $\mathbb{P}((U,G)\notin\mathcal{F}_n)\leq\varphi(n)$, then $\mathbb{P}((\mathbf{U},G)\notin\mathcal{F}_n)\leq\varphi(n)+o(1)$.
\label{cl:from_random_to_deterministic}
\end{claim}

\begin{proof}
It suffices to prove that 
$$
 \mathbb{P}\left((\mathbf{U},G)\in\mathcal{F}_n,\,|\mathbf{U}|\in\mathcal{J}_n\right)\geq 1-\varphi(n)-o(1).
$$
Let $\mathbf{U}_m$ be a uniformly random $m$-subset of $[n]$, independent of $G$. We get
\begin{align*}
 \mathbb{P}((\mathbf{U},G)\in\mathcal{F}_n,\,|\mathbf{U}|\in\mathcal{J}_n) & =\sum_{m\in\mathcal{J}_n} \mathbb{P}((\mathbf{U},G)\in\mathcal{F}_n\mid|\mathbf{U}|=m){n\choose m}2^{-n}\\
 & =\sum_{m\in\mathcal{J}_n} \mathbb{P}((\mathbf{U}_m,G)\in\mathcal{F}_n){n\choose m}2^{-n}\\
 &\geq\min_{m\in\mathcal{J}_n}\mathbb{P}((\mathbf{U}_m,G)\in\mathcal{F}_n)-o(1),
\end{align*}
due to the Chernoff bound (see, e.g.,~\cite[Theorem 2.1]{Janson}). Finally,
\begin{align*}
 \mathbb{P}((\mathbf{U},G)\in\mathcal{F}_n,\,|\mathbf{U}|\in\mathcal{J}_n) &\geq\min_{m\in\mathcal{J}_n}\sum_{U\in{[n]\choose m}}\mathbb{P}((\mathbf{U}_m,G)\in\mathcal{F}_n,\,\mathbf{U}_m=U)-o(1)\\
 &=\min_{m\in\mathcal{J}_n}\sum_{U\in{[n]\choose m}}\mathbb{P}((U,G)\in\mathcal{F}_n)\mathbb{P}(\mathbf{U}_m=U)-o(1)\\
 &\geq(1-o(1))(1-\varphi(n))-o(1)\geq 1-\varphi(n)-o(1).
\end{align*}

\end{proof}

\section{Proof of Theorem~\ref{th:bin_dense}, part 1}
\label{sc:proof_1}

We separately prove the upper and the lower bounds. Let 
\begin{equation}
\alpha_n=n(1-2p(1-p)),\quad \beta_n=\sqrt{8np(1-p)(1-2p(1-p))\ln n}.
\label{eq:alpha_beta}
\end{equation}

\paragraph{Lower bound.} Let us call a set $U\subset[n]$ {\it unique}, if there is no other subset $U'\subset[n]$ such that $G[U]\cong G[U']$. It suffices to prove that whp there are at most $2^{\alpha_n+\beta_n(1+o(1))}$ sets that are {\it not} unique. Let us first count sets $U$ such that there exists a set $U'$ with the following properties: 
\begin{itemize}
\item $|U|=|U'|$; 
\item $|U\cap U'|=|U|-1$; and 
\item the bijection $U\to U'$ that does not move any vertex from $U\cap U'$ is an isomorphism between $G[U]$ and $G[U']$.
\end{itemize}
 All pairs $(U,U')$ of such sets can be constructed as follows: choose vertices $x,x'$ arbitrarily. Let $N^{+}:=N^+(x,x')$ be their common neigbourhood, and let $N^{-}:=N^-(x,x')$ be the set of vertices that are adjacent neither to $x$, nor to $x'$. Let $W\subset N^{+}\cup N^{-}$. Then take $U=W\cup\{x\}$, $U'=W\cup\{x'\}$. Clearly, the number of such pairs $(U,U')$ is at most $n^2 2^{\xi}$, where $\xi$ is the maximum cardinality of $N^{+}\cup N^{-}$ over all $x$ and $x'$.

\begin{claim}
Whp $\xi=\alpha_n+\beta_n(1-o(1))$.
\label{cl}
\end{claim} 

\begin{proof}
Fix a small $\varepsilon>0$. Fix vertices $x,x'$. Then $\xi_{x,x'}:=|N^+\cup N^-|\sim\mathrm{Bin}(n,q:=1-2p(1-p))$. Then, by the de Moivre--Laplace limit theorem
\begin{align*}
 \mathbb{P}\left(\xi_{x,x'}-nq>\sqrt{(4+\varepsilon)nq(1-q)\ln n}\right)
 &=(1+o(1))\int_{\sqrt{(4+\varepsilon)\ln n}}^{\infty}\frac{1}{\sqrt{2\pi}}e^{-(1/2+o(1))t^2}dt\\
 &=\frac{n^{-2-\varepsilon+o(1)}}{\sqrt{2(4+\varepsilon)\pi\ln n}}.
\end{align*}
Since $\xi=\max_{x,x'}\xi_{x,x'}$, the union bound over all pairs $x,x'$ completes the proof of the upper bound 
$$
\xi\leq nq+\sqrt{(4+\varepsilon)nq(1-q)\ln n}=\alpha_n+\sqrt{1+\varepsilon/4}\cdot\beta_n
$$ 
whp.

In the same way, letting 
$$
\mathcal{B}_{x,x'}=\left\{\xi_{x,x'}-\alpha_n>\sqrt{(4-\varepsilon)nq(1-q)\ln n}\right\},
$$
we get
$$
 \mathbb{P}\left(\mathcal{B}_{x,x'}\right)=\frac{n^{-2+\varepsilon+o(1)}}{\sqrt{2(4+\varepsilon)\pi\ln n}}.
$$
Let $X$ be the number of pairs $\{x,x'\}$ such that the event 
$\mathcal{B}_{x,x'}(k)$ holds. We get that $\mathbb{E}[X]=n^{\varepsilon-o(1)}$. One can expect that $\mathrm{Var}[X]=o((\mathbb{E}[X])^2)$ implying the desired assertion due to Chebyshev's inequality.  However, the rare event that there exists a vertex with a large degree contributes superfluously to the variance (cf.~\cite{Rodi}). This makes the proof technically involved, and we move it to Appendix~\ref{sc:appendix}. For the sake of clarity of presentation, here we only prove a weaker version of the lower bound in Claim~\ref{cl} that asserts the right order of magnitude of the second-order term of $\xi$. This weaker version is stated below as Claim~\ref{cl:weaker}.
\end{proof}

\begin{claim}
Whp $\xi>\alpha_n+\frac{1+o(1)}{2}\beta_n$.
\label{cl:weaker}
\end{claim}

\begin{proof}
Fix a vertex $x$ and expose its neighbourhood $N(x)$. Let $m=|N(x)|$.  By the Chernoff bound, whp 
$$
|m-np|\leq\sqrt{np\ln\ln n}.
$$
Then, take any vertex $x'\notin N(x)$. It has
$ |N^+(x,x')\cup N^-(x,x')|=\xi^+_{x'}+\xi^-_{x'}$, where $\xi^+_{x'}\sim\mathrm{Bin}(m,p)$, $\xi^-_{x'}\sim\mathrm{Bin}(n-m-2,1-p)$ are mutually independent for all $x'\notin N(x)$. Let $x'$ be a vertex with the maximum value of $\xi^-_{x'}$. It is known that whp 
$$
\xi^-_{x'}=n(1-p)^2+(1-o(1))\sqrt{2n(1-p)^2p\ln n}
$$ 
(see, e.g.,~\cite{Bin_max}). Since $\xi^+_{x'}$ is independent of all $\xi^-_{y}$, $y\notin N(x)\cup\{x\}$, we get that whp $|\xi^+_{x'}-np^2|\leq\sqrt{np^2\ln\ln n}$. Therefore, whp
\begin{align*}
 \max_{y}|N^+(x,y)\cup N^-(x,y)| & \geq |N^+(x,x')\cup N^-(x,x')|\\
 & \geq n((1-p)^2+p^2)+\sqrt{(2-o(1))n(1-p)^2p\ln n}-\sqrt{np^2\ln\ln n}\\
 &\geq \alpha_n+\frac{1+o(1)}{2}\beta_n,
\end{align*}
completing the proof.
\end{proof}

Assume that $U,U'\subset[n]$ are such that $|U|=|U'|=:a$ and there exist a set $W\subset U\cap U'$ of size at least $a-10$ and an isomorphism $G[U]\to G[U']$ that does not move any vertex of $W$. Let the isomorphism map a vertex $x\in U\setminus U'$ to a vertex $x'\in U'$. Then sets $W\cup\{x\}$ and $W\cup\{x'\}$ induce isomorphic subgraphs of $G$. Therefore, due to Claim~\ref{cl},  whp the number of such pairs $\{U,U'\}$ is at most 
$$
n^{20}n^2 2^{\xi}\leq 2^{\alpha_n+\beta_n(1+o(1))}
$$ 
whp.

Let us call a set $U\subset[n]$ {\it bad}, if there exists another subset $U'\subset[n]$ and an isomorphism $G[U]\to G[U']$ that moves at least $10$ vertices of $U$.  It remains to show that the number of bad sets is $O_P(2^{\alpha_n})$.

Let $\mathbf{U}\subset[n]$ be a uniformly random subset. It suffices to prove that, for all large enough $n$,
$$
\mathbb{P}(\mathbf{U}\text{ is bad})\leq 2^{-2np(1-p)}.
$$
Indeed, assuming that, for some $\delta>0$ and for infinitely many $n$, 
$$
\mathbb{P}\left(\text{the number of bad sets is more than }\frac{1}{\delta}2^{\alpha_n}\right)>\delta,
$$
we get that 
$$
\mathbb{P}(\mathbf{U}\text{ is bad})>\delta\frac{\frac{1}{\delta}2^{\alpha_n}}{2^n}=2^{-2np(1-p)}
$$
 --- a contradiction. 
 
Since 
$$
\mathbb{P}(|\mathbf{U}-n/2|>0.4n)=\mathbb{P}(|\mathrm{Bin}(n,1/2)-n/2|>0.4n)\leq n\cdot {n\choose \lceil 0.1n\rceil} 2^{-n}\ll 2^{-n/2}\leq 2^{-2np(1-p)},
$$
it suffices to prove that 
$$
 \mathbb{P}(\mathbf{U}\text{ is bad},\,|\mathbf{U}|\in[0.1 n,0.9n])\leq  2^{-2np(1-p)-1}.
$$
Let us fix an integer $m\in[0.1 n,0.9n]$. We get
\begin{equation}
 \mathbb{P}(\mathbf{U}\text{ is bad},\,|\mathbf{U}|\in[0.1 n,0.9n])=\sum_{m=0.1n}^{0.9n} \mathbb{P}(\mathbf{U}\text{ is bad}\mid|\mathbf{U}|=m){n\choose m}2^{-n}.
 \label{eq:binomial_to_uniform_very_dense}
\end{equation}
We then denote by $\mathbf{U}_m$ a uniformly random $m$-subset of $[n]$. Since $G$ and $\mathbf{U}_m$ are sampled independently, we may treat $\mathbf{U}_m$ as a deterministic fixed subset $U\subset [n]$ of size $m$. 

Fix $U\in{[n]\choose m}$. Let $k\in[2m-n,m-1]$ be a non-negative integer. Fix a set $U'\subset[n]$ that has $k$ common vertices with $U$. Also, let us fix a subset $W\subset U\cap U'$ of size $j\leq k$ and a bijection $\sigma: U\to U'$ with the set of fixed points $W$. Let us estimate the probability that $\sigma$ is an isomorphism between $G[U]$ and $G[U']$. The bijection $\sigma$ can be represented as a disjoint union of directed paths $P_1=(x^1_1,\ldots,x^1_{s_1}),\ldots,P_{\ell}=(x^{\ell}_1,\ldots,x^{\ell}_{s_{\ell}})$ and cycles $C_1=(y^1_1,\ldots,y^1_{t_1}),\ldots,C_h=(y^h_1,\ldots,y^h_{t_h})$. Each cycle is entirely inside $(U\cap U')\setminus W$, and each path has the first vertex in $U\setminus U'$, the intermediate vertices in $(U\cap U')\setminus W$, and the last vertex in $U'\setminus U$. For every $i\in[\ell]$, adjacencies between $x^i_1$ and all the other vertices of $G$ identify uniquely the adjacencies from $x^i_2,\ldots, x^i_{t_i}$; the same holds for each cycle $C_i$. Thus, as soon as the edges of $G$ from $x^1_1,x^2_1,\ldots, x^{\ell}_1$ are exposed, all the edges with at least one vertex in $U'\setminus U$ are defined uniquely. Furthermore, adjacencies to $y^1_1,y^2_1,\ldots, y^{h}_1$ identify the rest of $E(G[U'])\setminus E(G[W])$. Recalling that $p\leq 1/2$ and letting $W':=\{y^1_1,y^2_1,\ldots, y^{h}_1\}$, we get
\begin{align*}
\mathbb{P}(\sigma(G[U])=G[U'])& \leq\left(\max\{p,1-p\}\right)^{\left|{U'\setminus W'\choose 2}\cup((U'\setminus U)\times(U\cap U'))\setminus{W\choose 2}\right|}\\
&\leq (1-p)^{{m-j-h\choose 2}+(k-j-h)j+(m-k)(j+h)},
\end{align*}
where, for brevity, we identify $(U'\setminus U)\times(U\cap U')$ with the set of unordered pairs $[(U'\setminus U)\times(U\cap U')]/S_2$. The last expression is maximised when $h$ is maximum possible, i.e. $h=\frac{k-j}{2}$ (here we assume that $h$ is not necessarily in integer by defining ${s\choose 2}:=\frac{s(s-1)}{2}$ for all real $s$). We finally get
\begin{align*}
 \mathbb{P}_{G}(U\text{ is bad}) & \leq\sum_{k,j}{m\choose k}{n-m\choose m-k}{k\choose j}(m-j)!(1-p)^{{m-j/2-k/2\choose 2}+(k-j)j/2+(m-k)(k+j)/2}\\
 & =\sum_{k=0.99m}^{m-1} {m\choose k}{n-m\choose m-k}\sum_{j=0}^{\min\{k,m-10\}}f(j)+e^{-\Theta(pm^2)},
\end{align*}
where 
$$
f(j)={k\choose j}(m-j)!(1-p)^{{m-j/2-k/2\choose 2}+\frac{(k-j)j}{2}+\frac{(m-k)(k+j)}{2}}.
$$
Indeed, for $k<0.99m$, recalling that $p\gg\frac{\ln n}{n}$,
$$
 {m\choose k}{n-m\choose m-k}f(j)<2^n\cdot 2^k\cdot m!\cdot (1-p)^{{m-k\choose 2}}=e^{-\Theta(m^2p)}.
$$

Since
$$
 \frac{f(j+1)}{f(j)}=\frac{k-j}{(j+1)(m-j)}(1-p)^{-3j/4+k/4-1/8},
$$
and 
$$
\frac{d}{d j}\left(\ln\frac{f(j+1)}{f(j)}\right)=-\frac{1}{k-j}-\frac{1}{j+1}+\frac{1}{m-j}+\frac{3}{4}\ln\frac{1}{1-p}>0
$$
when $j\in[k/6,k/2]$, we get that $ \frac{f(j+1)}{f(j)}<1$ when $j\leq k/6$, $ \frac{f(j+1)}{f(j)}>1$ when $j\geq k/2$, and $ \frac{f(j+1)}{f(j)}$ increases in $[k/6,k/2]$. Therefore, there exists $j^*$ such that $f$ decreases when $j<j^*$ and $f$ increases when $j>j^*$.

We get that, for some constant $c>0$,
\begin{align*}
\max_j f(j) & =\max\{f(0),\min\{f(k),f(m-10)\}\}\\
&=\min\left\{(m-k)!(1-p)^{{m-k\choose 2}+k(m-k)},c{k\choose m-10}(1-p)^{5m+k(m-k)/2}\right\}+e^{-\Theta(m^2p)}
\end{align*}
since
$$
 f(0)=m!(1-p)^{\frac{(m^2-3k^2/4)-(m-k/2)}{2}}\leq m! e^{-(1/8-o(1))m^2}=e^{-\Theta(m^2p)}.
$$

Thus, 
$$
 \mathbb{P}_{G}(U\text{ is bad})  \leq
 \sum_{k=0.99m}^{m-10} g_1(k)+10c\sum_{k=m-9}^{m-1}
 g_2(k)+e^{-\omega(n\ln n)},
$$
where
\begin{align*}
g_1(k) & ={m\choose k}{n-m\choose m-k}(k+1)(m-k)!(1-p)^{{m-k\choose 2}+k(m-k)};\\
g_2(k) & ={m\choose k}{n-m\choose m-k}{k\choose m-10}(1-p)^{5m+k(m-k)/2}.
\end{align*}
As above, we find the maximum value of $g_1(k)$ and $g_2(k)$. Recalling that $k\geq 0.99m$ and $m=\Theta(n)$, we get
$$
 \frac{g_1(k+1)}{g_1(k)}=\frac{(k+2)(m-k)}{(k+1)^2(n-2m+k+1)}(1-p)^{-k}>\frac{e^{pk}}{n^3}=n^{\omega(1)}>1, \text{ and}
$$
$$
 \frac{g_2(k+1)}{g_2(k)}=\frac{(m-k)^2}{(k-m+11)(n-2m+k+1)}(1-p)^{(m-2k-1)/2}>\frac{e^{0.49mp}}{n^2}=n^{\omega(1)}>1,
$$
implying
\begin{align*}
 \mathbb{P}_{G}(U\text{ is bad}) & \leq m^{12}n^{10}(1-p)^{{10\choose 2}+10(m-10)}+10cm^{10}n(1-p)^{5m+(m-1)/2}+e^{-\omega(n\ln n)}\\
 &=(1-p)^{(11/2-o(1))m}.
\end{align*}
Substituting the obtained bound on $\mathbb{P}_{G}(U\text{ is bad})=\mathbb{P}(\mathbf{U}_m\text{ is bad})$ into~\eqref{eq:binomial_to_uniform_very_dense}, we get
\begin{align*}
 \mathbb{P}(\mathbf{U}\text{ is bad},\,|\mathbf{U}|\in[0.1 n,0.9n]) &=
 \sum_{m=0.1n}^{0.9n}{n\choose m}2^{-n}\cdot\mathbb{P}(\mathbf{U}_m\text{ is bad})\\
 &\leq \sum_{m=0}^{n} (1-p)^{(11/2-o(1))m}{n\choose m}2^{-n}\\
 &=\left(\frac{1+(1-p)^{11/2-o(1)}}{2}\right)^n\leq 2^{-2np(1-p)},
\end{align*}
due to the following claim.

\begin{claim}
$1+(1-p)^{5.4}<2^{1-2p(1-p)}$ for all $p\in(0,1/2]$.
\label{cl:boring}
\end{claim}

The proof of Claim~\ref{cl:boring} is straightforward and technical, thus it is moved to Appendix~\ref{app:cl_boring_proof}.

\paragraph{Upper bound.}

Due to Claim~\ref{cl}, whp there exist vertices $x,x'$ such that the set $N^+(x,x')\cup N^-(x,x')$ of vertices that are either adjacent to both $x,x'$ or non-adjacent to both $x,x'$ has cardinality $\alpha_n+\beta_n(1-o(1))$. Clearly, for every $W\subset N^+(x,x')\cup N^-(x,x')$, sets $U:=W\cup\{x\}$ and $U':=W\cup\{x'\}$ induce isomorphic subgraphs in $G$, completing the proof.\\

\begin{remark}
From the proof it immediately follows that whp, for almost all {\it non}-unique sets $U$ (i.e., such that there exists $U'\neq U$ satisfying $G[U]\cong G[U']$), there exist two vertices $x,x'$ and a set $W\subset U$ such that $|U\setminus W|\leq 10$ and any vertex of $W$ is either a common neighbour  of $x,x'$, or a common non-neighbour of $x,x'$. It is easy to see that the stronger property holds for any such pair $\{U,U'\}$: either there exist $x,x'$ such that $U=W\cup\{x\}$, $U'=W\cup\{x'\}$, where $W\subset N^+(x,x')\cup N^-(x,x')$, or there exist $x,x',x''$ such that $U=W\cup\{x,x''\}$, $U'=W\cup\{x',x''\}$, where $W\subset N^+(x,x',x'')\cup N^-(x,x',x'')$, or there exist $x_1,x'_1,x_2,x'_2$ and $W\subset [N^+(x_1,x'_1)\cup N^-(x_1,x'_1)]\cap[N^+(x_2,x'_2)\cup N^-(x_2,x'_2)]$ such that $W\subset U$ and $|U\setminus W|\leq 10$. It is also possible to show (in the same way as in the proof of the non-existence statement in Claim~\ref{cl}) that the number of sets $U$ of the second and the third type is $o\left(2^{\alpha_n}\right)$ whp. Thus, whp, for almost all non-unique sets $U$ there exist $x,x'$ such that $U\setminus\{x\}\subset N^+(x,x')\cup N^-(x,x').$
\end{remark}

\section{Proof of Theorem~\ref{th:bin_dense}, part 2}
\label{sc:proof_2}

Since we have already proved the first part of Theorem~\ref{th:bin_dense}, we may assume that $p\leq\frac{C\ln n}{n}$ for some $C>0$.

\begin{claim}
Let $U\subset[n]$ have size $(1/2+o(1))n$. Then, whp, for every subset $W\subset U$, the number of edges in $U$ incident to at least one vertex in $W$ is at least $|W|\ln\ln n$.
\end{claim}

\begin{proof}
The assertion immediately follows from the fact that whp every vertex in $U$ has at least $\varepsilon^2\ln n$ neighbours in $U$. To prove the latter, let $|U|=m$.
\begin{align*}
 \mathbb{P}\left(\exists x\in U:\,|N(x)\cap U|<\varepsilon^2\ln n\right) &\leq m\mathbb{P}\left(\mathrm{Bin}(m-1,p)<\varepsilon^2\ln n\right)\\
 &=m\sum_{i<\varepsilon^2\ln n}{m\choose i}p^i(1-p)^{m-1-i}\\
&\stackrel{(*)}\leq m\ln n \left(\frac{emp}{\varepsilon^2\ln n}\right)^{\varepsilon^2\ln n} e^{-p(m-\ln n)}\\
&\stackrel{(**)}\leq m\cdot \exp\left[-\left(1+\varepsilon/2+o(1)-\varepsilon^2\ln\left(\frac{(1+\varepsilon/2)e}{\varepsilon^2}\right)\right)\ln n\right]\\
&\leq m\cdot n^{-1-\varepsilon/4}=o(1)
\end{align*}
for small enough $\varepsilon$. The inequality (*) follows from the fact that ${m\choose i}p^i$ increases in $i$. The inequality (**) follows from the fact that $ p^{\varepsilon^2\ln n} e^{-pm}$ decreases in $p$.
\end{proof}

Let us call a set $U\subset[n]$ {\it unique}, if there is no other subset $U'\subset[n]$ such that $G[U]$ and $G[U']$ are isomorphic.

Let $\mathbf{U}\subset[n]$ be a uniformly random subset. It suffices to prove that $\mathbb{P}(\mathbf{U}\text{ is non-unique})=o(1)$. Indeed, assuming that, for some $\delta>0$ and for infinitely many $n$, 
$$
\mathbb{P}\left(\text{the number of non-unique sets is more than } \delta 2^{n}\right)>\delta,
$$
we get that 
$$
\mathbb{P}(\mathbf{U}\text{ is non-unique})>\delta\frac{\delta 2^n}{2^n}=\delta^2
$$
 --- a contradiction. 
 
Let $m\in\mathcal{J}_n$, $U\in{[n]\choose m}$. Due to Claim~\ref{cl:from_random_to_deterministic}, it suffices to prove that $\mathbb{P}_G(U\text{ is non-unique})=o(1)$ uniformly over $m$ and $U$. Arguing in the same way as in the proof of part 1 of Theorem~\ref{th:bin_dense} and using the fact that whp, for every subset $W\subset U$, the number of edges in $E(G[U])\setminus E(G[W])$ is at least $|W|\ln\ln n$, we get
\begin{align*}
 \mathbb{P}_{G}(U\text{ is non-unique})=o(1)+\sum_{k=\max\{0,2m-n\}}^{m-1} \sum_{j=0}^{k} g(k,j) = o(1)+
 \sum_{j=0}^{m-1}\sum_{k=\max\{j,2m-n\}}^{m-1}  g(k,j) 
\end{align*}
where 
$$
 g(k,j)= {m\choose k}{n-m\choose m-k} {k\choose j}(m-j)!(1-p)^{{m-j/2-k/2\choose 2}+\frac{(k-j)j}{2}+\frac{(m-k)(k+j)}{2}}\left(\frac{p}{1-p}\right)^{\left(m-\frac{j+k}{2}\right)\ln\ln n}.
$$
For a fixed $j$,
$$
 \frac{g(k+1,j)}{g(k,j)}=\frac{(m-k)^2}{(n-2m+k+1)(k+1-j)}(1-p)^{\frac{j-3k-1/2}{4}}\left(\frac{1-p}{p}\right)^{\ln\ln n/2}>\frac{p^{-\ln\ln n/2}}{n^2}>1.
$$
Thus, for every $j$ and every $k$, $g(k,j)\leq g(m-1,j)$. Then
\begin{align*}
 \mathbb{P}_{G}(U\text{ is non-unique})=o(1)+ n\sum_{j=0}^{m-1}g(m-1,j) .
\end{align*}
Now, observe that, for all $j$,
\begin{align*}
 \frac{g(m-1,j+1)}{g(m-1,j)} &=\frac{m-1-j}{(j+1)(m-j)}(1-p)^{-3j/4+m/4-5/8}\left(\frac{1-p}{p}\right)^{\ln\ln n/2}\\
 &>\frac{(1-p)^{m/4}}{n^2}p^{-\ln\ln n/2}>n^{-C+\ln\ln n/2}>1.
\end{align*}
 We finally get
\begin{align*}
 \mathbb{P}_{G}(U\text{ is non-unique}) & \leq o(1)+n^2g(m-1,m-1)\\
 &\leq o(1)+n^2(m-1)(n-m)(1-p)^{m-1-\ln\ln n} p^{\ln\ln n}\\
 &\leq o(1)+n^4 p^{\ln\ln n}=o(1),
\end{align*}
completing the proof.

\section{Proof of Theorem~\ref{th:bin_dense}, part 3}
\label{sc:proof_3}

Let us call a set $U\subset[n]$ {\it bad}, if there exist at least $\ln n$ sets $U'\subset[n]$ such that $G[U]$ and $G[U']$ are isomorphic.

Let $\mathbf{U}\subset[n]$ be a uniformly random subset. It suffices to prove that $\mathbb{P}(\mathbf{U}\text{ is bad})=1-o(1)$. Indeed, assume that with probability at least $\delta>0$ the number of non-isomorphic induced graphs is at least $\delta\cdot 2^n$. Then
$$
 \mathbb{P}(\mathbf{U}\text{ is not bad})\geq\delta\cdot\frac{\delta\cdot 2^n-2^n/\ln n}{2^n}=\delta^2-o(1) 
$$
--- a contradiction.

Let $m\in\mathcal{J}_n$, $U\in{[n]\choose m}$. Due to Claim~\ref{cl:from_random_to_deterministic}, it suffices to prove that $\mathbb{P}_G(U\text{ is bad})=1-o(1)$ uniformly over $m$ and $U$. The following claim completes the proof.


Let $\eta$ be the number of vertices from $U$ that do not have neighbours in $U$ and let $\eta'$ be the number of vertices outside $U$ that do not have neighbours in $U$.

\begin{claim}
\label{cl:two_isolated_vertices}
Whp $\eta\geq 1$ and $\eta'\geq\ln n$.
\end{claim}

Indeed, let $x\in U,x'\notin U$ be vertices that do not have neighbours in $U$. Then $G[U\cup\{x\}\setminus\{x'\}]\cong G[U]$. In this way, we can get at least $\ln n$ sets that induce graphs isomorphic to $G[U]$.

\begin{proof}[Proof of Claim~\ref{cl:two_isolated_vertices}.] Since $G[U]\sim G(|U|,p)$, the fact that whp $\eta\geq 1$ is known (see, e.g.,~\cite[Theorem 3.5]{Bol}).  Let $m=|U|$. Then $\eta'\sim\mathrm{Bin}(n-m,(1-p)^m)$. Therefore, 
$$
\mathbb{E}\eta'=(n-m)(1-p)^m\geq
\frac{n}{2}e^{-(1-\varepsilon/2+o(1))\ln n}=n^{\varepsilon/2-o(1)}.
$$
The claim follows from the fact that a binomial random variable with growing expectation is concentrated around its expectation.
\end{proof}

\section{Proof of Theorem~\ref{th:bin_dense}, part 4}
\label{sc:proof_4}

Without loss of generality, we assume that $\varepsilon>0$ is small enough. Let us call a set $U\subset[n]$ {\it bad}, if there exist at least $2^{\varepsilon n}$ sets $U'\subset[n]$ such that $G[U]$ and $G[U']$ are isomorphic.

Let $\mathbf{U}\subset[n]$ be a uniformly random subset. It suffices to prove that $\mathbb{P}(\mathbf{U}\text{ is bad})=o(1)$. Indeed, assume that with probability at least $\delta>0$ the number of non-isomorphic induced graphs is at most $\delta 2^{(1-\varepsilon)n}$. Then
$$
 \mathbb{P}(\mathbf{U}\text{ is bad})\geq\delta\cdot\frac{2^n-\delta 2^{(1-\varepsilon)n} 2^{\varepsilon n}}{2^n}=\delta(1-\delta) 
$$
--- a contradiction.

Let $m\in\mathcal{J}_n$, $U\in{[n]\choose m}$. Due to Claim~\ref{cl:from_random_to_deterministic}, it suffices to prove that $\mathbb{P}_{G}(U\text{ is bad})=o(1)$ uniformly over $m$ and $U$. 


\begin{claim}
Whp every subset $W\subset U$ of size at least $\varepsilon^2 n$ spans the number of edges satisfying $||E(G[W])|-\mathbb{E}[|E(G[W])|]|\leq\frac{1}{2}\mathbb{E}[|E(G[W])|]$.
\label{cl:concentrated_subgraphs}
\end{claim}

\begin{proof}
Let $W\subset U$ have size $y\geq\varepsilon^2 n$. It induces $|E(G[W])|\sim\mathrm{Bin}({y\choose 2},p)$ edges with 
$$
\mathbb{E}[|E(G[W])|]={y\choose 2}p\sim\frac{y^2}{2}p>\frac{C\varepsilon^2}{2} y.
$$
By the Chernoff bound, 
$$
\mathbb{P}\left(||E(G[W])|-\mathbb{E}[|E(G[W])|]|>\frac{1}{2}\mathbb{E}[|E(G[W])|]\right)\leq e^{-10y/\varepsilon^2}
$$
for large enough $C=C(\varepsilon)$. By the union bound,
$$
 \mathbb{P}\left(\exists W\,\, ||E(G[W])|-\mathbb{E}[|E(G[W])|]|>\frac{1}{2}\mathbb{E}[|E(G[W])|]\right)\leq 2^n e^{-10n}=o(1).
$$
\end{proof}


  Consider the family $\mathcal{F}_m$ of all sets of size $m$ that have subgraphs with the number of edges concentrated as in Claim~\ref{cl:concentrated_subgraphs}. We know that whp $U\in\mathcal{F}_m$. Let $\mathcal{U}(U)$ be the family of all $m$-subsets of $[n]$ that have less than $m-\varepsilon^2 n$ common vertices with $U$. Then if $U\in\mathcal{F}_m$, for every $U'\in\mathcal{U}(U)$ such that $G[U']\cong G[U]$, there should exist an isomorphism $\varphi: U\to U'$ that moves all edges from $E(G[U\setminus U'])$. Since $|U\setminus U'|\geq\varepsilon^2n$ and $U\in\mathcal{F}_m$, there are at least $\left(\frac{1}{4}-o(1)\right)\varepsilon^4 n^2p$ edges moved my $\varphi$. Then
\begin{align}
 \mathbb{P}\biggl(\exists U'\in\mathcal{U}(U)\,\, G[U]\cong G[U']\biggr)
 &\leq \mathbb{P}(U\notin\mathcal{F}_m)+2^n n! p^{(\frac{1}{4}-o(1))\varepsilon^4 n^2p}\notag\\
 &\leq o(1)+\exp\left (n\ln n-\left(\frac{1}{4}-o(1)\right)\varepsilon^4 Cn\ln n\right)=o(1),
 \label{eq:binomial_p4_UU'}
\end{align}
for large enough $C=C(\varepsilon)$.
Therefore, whp all subsets $U'\subset[n]$ such that $G[U]\cong G[U']$ have at least $m-\varepsilon^2n$ common vertices with $U$. Thus, whp the number of such sets is at most
$$
 \sum_{k\geq m-\varepsilon^2 n}{m\choose k}\leq n{m\choose\lceil m-\varepsilon^2 n\rceil}<2^{\varepsilon n},
$$
completing the proof.

\section{Proof of Theorem~\ref{th:bin_dense}, part 5}
\label{sc:trees}

In Section~\ref{sc:GW}, we prove that the expected number of non-isomorphic subtrees in a $\mathrm{Pois}(1-\varepsilon)$--Galton--Watson tree equals $\exp(\Omega(1/\varepsilon))$. In Section~\ref{sc:anatomy}, we recall the structure of the automorphism group of $G(n,p)$ as well as the contiguous model of the giant component due to Ding, Lubetzky, and Peres~\cite{DLP}. We then combine these results in Section~\ref{sc:part5_completing} and prove the lower bound from part 5 of Theorem~\ref{th:bin_dense}. The upper bound is much more straightforward, as explained in Section~\ref{sc:Intro}, and its proof is presented in Section~\ref{sc:part5_upper}.


\subsection{Galton--Watson trees}
\label{sc:GW}

For a rooted tree $T$ with root $R$, let $f(T)$ be the number of non-isomorphic subtrees of $T$ rooted in $R$. For a vertex $v$ of $T$, let $T_v$ be the subtree of $T$ induced by all descendants of $v$, including $v$ itself. For technical reasons, it will be more convenient to work with $f_+(T):=f(T)+1$. Observe that if $R$ has children $v_1,\ldots,v_j$, then 
$$
 f_+(T)\geq\frac{\prod_{i=1}^j f_+(T_{v_i})}{j!}+1.
$$
Indeed, if, for every $i\in[j]$, we choose a pair of (not necessarily non-empty) non-isomorphic trees $T_i\subset T_{v_i}$, $T'_i\subset T_{v_i}$ rooted in $v_i$, then the rooted trees $T_0$ and $T'_0$ obtained by adding edges between $R$ and $v_i$ in every $T_i$ and $T'_i$ respectively are isomorphic if and only if the multisets of unlabelled trees $\{\!\!\{T_i,i\in[j]\}\!\!\}$ and $\{\!\!\{T'_i,i\in[j]\}\!\!\}$ coincide.  

Let $\mathbf{T}$ be a rooted Galton--Watson tree with offspring distribution $\mathrm{Pois}(1-\varepsilon)$. 
\begin{claim}
For all small enough $\varepsilon>0$, $\mathbb{E}\ln f(\mathbf{T})>\frac{0.003}{\varepsilon}$. 
\label{cl:GW-main}
\end{claim}
\begin{proof}
Let $X=\ln f_+(\mathbf{T})$, and let $\xi$ be the number of children of the root. Then
\begin{align*}
 \mathbb{E}X & =\sum_{j=0}^{\infty}\mathbb{E}(X\mid \xi=j)\cdot\mathbb{P}(\xi=j)
 \geq\ln 2\cdot\mathbb{P}(\xi=0)+\sum_{j=1}^{\infty}\mathbb{E}\ln((f_+(\mathbf{T}))^j/j!+1)\cdot\mathbb{P}(\xi=j)\\
 &=
 \ln 2\cdot\mathbb{P}(\xi=0)+\sum_{j=1}^{\infty}\biggl(j\cdot\mathbb{E}X+\mathbb{E}\ln\left(1+\frac{j!}{(f_+(\mathbf{T}))^j}\right)-\ln(j!)\biggr)\cdot\mathbb{P}(\xi=j)\\
 &\geq
 \ln 2\cdot \mathbb{P}(\xi=0)+\mathbb{E}X\cdot\mathbb{E}\xi+\mathbb{P}(\xi=1)\cdot\mathbb{E}\ln\left(1+\frac{1}{f_+(\mathbf{T})}\right)-\sum_{j=1}^{\infty}\ln(j!)\cdot\mathbb{P}(\xi=j).
\end{align*}
Clearly, denoting by $Y$ the total progeny of the Galton--Watson process, we get
\begin{align*}
 \mathbb{E}\ln\left(1+\frac{1}{f_+(\mathbf{T})}\right) &\geq \ln(3/2)\cdot\mathbb{P}(Y=1)+\ln(4/3)\cdot\mathbb{P}(Y=2)\\
 &=\ln(3/2)\cdot\mathbb{P}(\xi=0)+\ln(4/3)\cdot\mathbb{P}(\xi=1)\cdot\mathbb{P}(\xi=0).
\end{align*}
Since $\mathbb{E}\xi(\xi-1)=(1-\varepsilon)^2$, it follows that
\begin{align*} 
 \mathbb{E}X &\geq\ln 2\cdot e^{-(1-\varepsilon)}+(1-\varepsilon)\mathbb{E}X+(1-\varepsilon)e^{-(1-\varepsilon)}\left(e^{-(1-\varepsilon)}\ln(3/2)+(1-\varepsilon)e^{-2(1-\varepsilon)}\ln(4/3)\right)\\
 &\quad\quad\quad\quad\quad\quad\quad\quad-(1-\varepsilon)^2+\sum_{j=2}^{\infty}\left[j(j-1)-\ln(j!)\right]\cdot\mathbb{P}(\xi=j).
 \end{align*}
Notice that, for every $j\geq 2$, $\ln(j!)<j(j-1)$. Therefore,
\begin{align*}
 \sum_{j=2}^{\infty}\left[j(j-1)-\ln(j!)\right]\cdot\mathbb{P}(\xi=j)&\geq
  \sum_{j=2}^{5}\left[j(j-1)-\ln(j!)\right]\frac{e^{-1+\varepsilon}(1-\varepsilon)^j}{j!}\\
  &=
  e^{-1}\left(\frac{2-\ln 2}{2}+\frac{6-\ln 6}{6}+\frac{12-\ln 24}{24}+\frac{20-\ln 120}{120}\right)+O(\varepsilon),
\end{align*}
implying
\begin{multline*}
 \mathbb{E}X\geq\frac{1}{\varepsilon}\biggl( e^{-1}\biggl(\ln 2 +\frac{2-\ln 2}{2}+\frac{6-\ln 6}{6}+\frac{12-\ln 24}{24}+\frac{20-\ln 120}{120}+e^{-1}\ln\frac{3}{2}+e^{-2}\ln\frac{4}{3}\biggr)\\
 -1+O(\varepsilon)\biggr)>\frac{0.004}{\varepsilon},
\end{multline*}
for small enough $\varepsilon>0$. Finally, we get
\begin{align*}
\mathbb{E}\ln f(\mathbf{T}) & =\mathbb{E}\ln\left((f(\mathbf{T})+1)\left(1-\frac{1}{f(\mathbf{T})+1}\right)\right)\\
&=\mathbb{E}X+\mathbb{E}\ln\left(1-\frac{1}{f_+(\mathbf{T})}\right)>\frac{0.004}{\varepsilon}-\ln 2>\frac{0.003}{\varepsilon}.
\end{align*}
\end{proof}

Since $\mathbb{E}\ln f(\mathbf{T})>\frac{0.003}{\varepsilon}$, there exists $A=A(\varepsilon)$ such that $\mathbb{E}[\ln f(\mathbf{T})\cdot\1(f(\mathbf{T})\leq A)]>\frac{0.002}{\varepsilon}$. Since the random variable $\tilde X:=\ln f(\mathbf{T})\cdot\1(f(\mathbf{T})\leq A)$ is bounded, $\mathbb{E}\tilde X^2<\infty$. Chebyshev's inequality implies the following.

\begin{lemma}
Let $\varepsilon>0$ be small enough, and let $\mathbf{T}_1,\ldots,\mathbf{T}_n$ be independent Galton--Watson trees with offspring distribution $\mathrm{Pois}(1-\varepsilon)$. Then, whp (as $n\to\infty$), $\prod_{i=1}^nf(\mathbf{T}_i)\geq e^{0.002n/\varepsilon}$.
\label{lm:GW-many}
\end{lemma}

\subsection{Anatomy of a strictly supercritical random graph}
\label{sc:anatomy}

A connected graph is called {\it complex} if it has at least two cycles. 

Let  $C>1+\varepsilon>1$ be constants, $(1+\varepsilon)\leq np\leq C$, $G\sim G(n,p)$, and let $H$ be the 2-core of the union of complex components of $G$. In the supercritical regime, we shall consider subgraphs of the giant component of $G_n$ that contain the entire 2-core. If two such induced subgraphs are isomorphic, then any isomorphism between them preserves the set of vertices of the 2-core. Thus, in order to prove that there are many such non-isomorphic subgraphs, we will use the fact that the 2-core is almost asymmetric. In~\cite{VZ}, Verbitsky and the second author obtained a full description of the automorphism group $\mathrm{Aut}(H)$ of $H$. In particular, they proved the following.

\begin{claim}[\cite{VZ}]
$|\mathrm{Aut}(H)|=O_P(1)$.
\label{cl:aut}
\end{claim}

We will also make use of the contiguous model, due to Ding, Lubetzky, and Peres~\cite{DLP}. Here, we formulate its simplified and weaker version which is precisely~\cite[Equation (5.9)]{DLP} and is enough for our goals. We also note that the results in~\cite{DLP} are presented for constant $np$, although for our goals we need non-constant $np$. Nevertheless, literally the same proof allows to get the following\footnote{The authors of~\cite{DLP} rely on the local limit theorem for a sequence of independent identically distributed random variables \cite[Theorem 2.2]{DLP}. For non-constant $np$, the proof requires a generalisation to triangular arrays that can be found, e.g., in~\cite{Rozanov}.}.

\begin{claim}[\cite{DLP}]
Let $\lambda>1$. 
  Let $1<np=\lambda+o(1)$ and let $\lambda'$ be the unique number in $(0,1)$ such that $\lambda' e^{-\lambda'}=np e^{-np}$. Let $G\sim G(n,p)$, $G'$ be the union of complex components of $G$, and $H$ be the 2-core of $G'$. Let $G''$ be obtained from $H$ by attaching (independently of $G$) an independent $\mathrm{Pois}(\lambda')$--Galton--Watson tree to each vertex of $H$. Then, for every set of unlabelled graphs $\mathcal{F}$, if $\mathbb{P}(G''\in\mathcal{F})=o(1)$ then $\mathbb{P}(G'\in\mathcal{F})=o(1)$.
\label{cl:DLP}
\end{claim}

\subsection{Completing the proof of the lower bound}
\label{sc:part5_completing}

For a graph $\Gamma$ on $V\subset\mathbb{N}$, we denote by $\mathrm{Core}(\Gamma)$ the 2-core of the union of complex components of $\Gamma$. For $v\in V(\mathrm{Core}(\Gamma))$, let $T_v(\Gamma)$ be the unlabelled version of the inclusion-maximal subtree of $\Gamma$ rooted in $v$ and sprouting from the 2-core, i.e. the connected component of $v$ in $\Gamma[V\setminus V(\mathrm{Core}(\Gamma))\cup\{v\}]$. Let {\it type} of $\Gamma$ be the tuple 
$$
\mathbf{t}(\Gamma):=(T_v(\Gamma),\,v\in V(\mathrm{Core}(\Gamma))),
$$
where the order of the elements in the tuple respects the order of integers in $V$.

Now, we can complete the proof of part 5 of Theorem~\ref{th:bin_dense}. Assume towards contradiction that there exist $p=p(n)$ and a constant $\delta>0$ such that $np\in[1+\varepsilon,C]$ and
\begin{equation}
\mathbb{P}\biggl(\mu(G_n)<\exp(\varepsilon n/1000)\biggr)>\delta
\label{eq:assumption_sparse}
\end{equation} 
for $G_n\sim G(n,p)$ and all $n$ large enough. Then, there exists an increasing sequence of positive integers $n_k$, $k\in\mathbb{N}$, such that $p(n_k)n_k\to\lambda$ for some $\lambda\in[1+\varepsilon,C]$. Since, for the subsequence $G_{n_k}$, we have that $\mathbb{P}(\mu(G_{n_k})<\exp(\varepsilon n_k/1000))>\delta$ for all $k$ large enough, in what follows, without loss of generality, we assume that $np=\lambda+o(1)$. In what follows, we also omit the subscript and write $G:=G_n$.

Let $\lambda'=\lambda'(n)$ be the sequence from the assumptions of Claim~\ref{cl:DLP}. Note that $\lambda'\leq 1-\varepsilon+\varepsilon^2$, since $\lambda'e^{-\lambda'}$ increases on $(0,1)$ and, for small enough $\varepsilon>0$,
$$
\lambda e^{-\lambda}<e^{-1}(1+\varepsilon)(1-\varepsilon+\varepsilon^2/2)<e^{-1}(1-\varepsilon+\varepsilon^2)(1+\varepsilon-\varepsilon^2/2)<(1-\varepsilon+\varepsilon^2)e^{-1+\varepsilon-\varepsilon^2}.
$$
We define the following set of connected graphs $\mathcal{F}$ on subsets of $\mathbb{N}$: $\Gamma\in\mathcal{F}$, if the number of distinct $\mathbf{t}(\tilde \Gamma)$ over all connected $\tilde \Gamma\subset\Gamma$ with $\mathrm{Core}(\Gamma)=\mathrm{Core}(\tilde \Gamma)$ is at least $\exp(|V(\mathrm{Core}(\Gamma))|/(500(1-\lambda')))$. Note that $\mathcal{F}$ is isomorphism-closed --- thus, it can be treated as a set of unlabelled graphs. 
  If $\lambda'<1$ is close enough to 1, then due to Lemma~\ref{lm:GW-many} and Claim~\ref{cl:DLP}, whp $G\in\mathcal{F}$. Let $G'$ be the union of complex components of $G$. Let $\tilde G_1,\tilde G_2\subset G'$ be isomorphic connected subgraphs with $\mathrm{Core}(\tilde G_1)=\mathrm{Core}(\tilde G_2)=\mathrm{Core}(G')$ and let $\varphi:\tilde G_1\to\tilde G_2$ be an isomorphism. Clearly, $\varphi(V(\mathrm{Core}(\tilde G_1)))=V(\mathrm{Core}(\tilde G_2))$. If $\varphi$ acts trivially on the 2-core, then $\mathbf{t}(\tilde G_1)=\mathbf{t}(\tilde G_2)$. Otherwise, $\varphi$ is a non-trivial automorphism of $H:=\mathrm{Core}(G')$. Due to Claim~\ref{cl:aut}, whp the number of automorphisms of $H$ is less than, say, $n$. Therefore, whp, $G'$ has at least $\frac{1}{n}\exp(|V(H)|/(500(1-\lambda')))$ non-isomorphic induced subgraphs. On the other hand, whp $|V(H)|=(1-\lambda'+o(1))(1-\lambda'/\lambda)n$~\cite{Pittel}. Therefore, whp 
$$
 \mu(G)\geq\exp((1-\lambda'/\lambda+o(1))n/500)>\exp(\varepsilon n/1000)
$$  
--- a contradiction with~\eqref{eq:assumption_sparse}.

Finally, let $\lambda'<1$ be not large enough so that Lemma~\ref{lm:GW-many} is not applicable. Since we may choose $\varepsilon$ as small as we need, we can assume that whp $|V(\mathrm{Core}(G))|>\sqrt{\varepsilon}\cdot n$ and that $\lambda'<1-\sqrt{\varepsilon}$. Since a $\mathrm{Pois}(1-\lambda')$--Galton--Watson tree is non-trivial with probability $1-e^{-(1-\lambda')}$, we get that whp $(1-e^{-(1-\lambda')}+o(1))$-fraction of vertices of $\mathrm{Core}(G)$ have a non-trivial tree growing from it. Therefore, due to Claim~\ref{cl:aut}, whp 
$$
\mu(G)\geq 2^{\sqrt{\varepsilon}(1-e^{-(1-\lambda')}+o(1)) n}\geq e^{\varepsilon n/10}
$$
--- a contradiction with~\eqref{eq:assumption_sparse}.

\subsection{Upper bound}
\label{sc:part5_upper}

Let $np\leq 1+\varepsilon$, $\varepsilon>0$ is small, and $G\sim G(n,p)$. 

Assume first that $np=1+\Theta(1)$. Let us recall that in this case whp $G$ has a single connected component of size $\Theta(n)$, and all the other components have size $O(\log n)$, see, e.g.,~\cite[Chapter 5]{Janson}. Let us now prove that whp the number of vertices in small components that have size at least $\ln\ln n$ is $o(n)$. Fix $v\in[n]$. By the union bound, the probability that $v$ belongs to a connected component of size $k$ is at most ${n-1\choose k-1} k^{k-2}p^{k-1}(1-p)^{k(n-k)}$. Let $X$ be the number of vertices  that belong to connected components of size $k\in[\ln\ln n,\ln^2n]$ in $G$. By linearity of expectation, we get
$$
 \mathbb{E}X\leq\sum_{k=\lceil\ln\ln n\rceil}^{\lfloor \ln^2n\rfloor}n{n-1\choose k-1} k^{k-2}p^{k-1}(1-p)^{k(n-k)}\leq \sum_{k=\lceil\ln\ln n\rceil}^{\lfloor \ln^2n\rfloor}\frac{1}{p}(e^{1-np}np)^k =\frac{n}{(\ln n)^{\Omega(1)}}.
$$
By Markov's inequality, whp $X=o(n)$. Therefore, whp $G$ is a disjoint union of graphs $G^{(1)},G^{(2)},G^{(3)}$, where $G^{(1)}$ is a connected graph of size $\Theta(n)$, $G^{(2)}$ has $o(n)$ vertices, and all components of $G^{(3)}$ have size at most $\ln\ln n$. The number of non-isomorphic graphs on at most $n$ vertices with all components of size at most $\ln\ln n$ is less than $n^{2^{(\ln\ln n)^2}}=e^{o(n)}$. Indeed,  there are at most $2^{(\ln\ln n)^2}$ non-isomorphic graphs on at most $\ln\ln n$ vertices, and to describe the isomorphism type of a graph on at most $n$ vertices with all connected components that small, it is enough to state to which isomorphism type the connected component containing each of the at most $n$ vertices belongs. The giant component $G^{(1)}$ has at most $(2\varepsilon+O(\varepsilon^2))n$ vertices whp~\cite[Theorem 5.4]{Janson}. Therefore, whp 
$$
\mu(G)\leq 2^{(2\varepsilon+O(\varepsilon^2))n+o(n)}<2^{3\varepsilon n},
$$
as required.

Actually the case $np\leq 1+o(1)$ also follows since the property of containing at least $2\varepsilon n$ vertices in components of size at least $\ln\ln n$ is increasing, and whp $G(n,1+\varepsilon/2)$ does not have it.

\section{Proof of Theorem~\ref{th:bin_dense}, part 6}
\label{sc:proof_6}

Let $np\leq 1-\varepsilon$ and $G\sim G(n,p)$. We will define explicitly a decreasing function $c_2=c_2(\varepsilon)\in(0,1)$ satisfying the required properties: whp $\mu(G)\leq 2^{n^{c_2}}$ and $c_2(1-)=0$.

First, let $\varepsilon\leq 0.99$. Fix any $c\in(0,1)$. Recall that whp $G$ does not contain components of size at least $\ln^2 n$ (see, e.g.,~\cite[Theorem 5.4]{Janson}). Observe that whp the number of connected components of size more than $\frac{2(1-c)}{\varepsilon^2}\ln n$ and less than $\ln^2 n$ is at most $n^{c}\ln n$ by Markov's inequality since the expected number of such components is at most
\begin{align*}
 \sum_{k=2(1-c)\ln n/\varepsilon^2}^{\ln^2 n}{n\choose k} k^{k-2} p^{k-1}(1-p)^{k(n-k)}&\leq
 \sum n^k e^k p^{k-1}e^{-pk(n-k)}\\
 &\leq (e+o(1))n\sum e^{(k-1)(1+\ln(np)-np)}\\
 &\leq (e+o(1))n\sum e^{(k-1)(\varepsilon+\ln(1-\varepsilon))}\\
 &=O\left(n\sum e^{-k\varepsilon^2/2}\right)=O(n^{c}).
\end{align*}
Recall that whp $G$ does not contain complex components (see, e.g.,~\cite[Theorem 5.5]{Janson}) and that the number of isomorphism classes of trees on $k$ vertices at most $k 4^k$~\cite{Otter}. Therefore, the number of isomorphism classes of unicyclic graphs on $k$ vertices is at most $k^3 4^k$. Let $U$ be the set of all vertices in connected components of $G$ of size more than $\frac{2(1-c)}{\varepsilon^2}\ln n$. Every subgraph of $G$ is a disjoint union of a subgraph of $G[U]$ and connected components of size at most $\frac{2(1-c)}{\varepsilon^2}\ln n$ whp. Since the number of components in every subgraph of $G$ is at most $n$, we get that the number of non-isomorphic subgraphs of $G[[n]\setminus U]$ consisting of connected components of size exactly $k$ is at most $n^{k^3\cdot 4^k}$, for $k\leq\frac{2(1-c)}{\varepsilon^2}\ln n$. We get that whp
$$
\mu(G)\leq 2^{|U|}\cdot\prod_{k\leq \frac{2(1-c)}{\varepsilon^2}\ln n}n^{k^3\cdot 4^k}.
$$ 
Therefore, letting $c=1-\varepsilon^2/100$, we get that whp 
$$
\mu(G)\leq n^{4^{\frac{2(1-c)}{\varepsilon^2}\ln n}\cdot \ln^4 n}\cdot 2^{n^c\cdot \ln^3 n}=2^{2n^{(\ln 4)/50}\cdot \ln^5 n+n^c\cdot \ln^3 n}=2^{n^{c+o(1)}}.
$$
Hence, we can take $c_2=1-\varepsilon^2/200$, say.

Second, let $\varepsilon>0.99$. In this case, in a similar way, by Markov's inequality we get that whp there are no components of size more than $\frac{2\ln n}{-\varepsilon-\ln(1-\varepsilon)}$. Indeed, the expected number of components of such size (and less than $\ln^2 n$) is at most
\begin{align*}
 \sum_{k=2\ln n/(-\varepsilon-\ln(1-\varepsilon))}^{\ln^2 n}{n\choose k} k^{k-2} p^{k-1}(1-p)^{k(n-k)}&\leq
 (e+o(1))n\sum e^{(k-1)(1+\ln(np)-np)}\\
 &\leq (e+o(1))n\sum e^{(k-1)(\varepsilon+\ln(1-\varepsilon))}\\
 &=O\left(n e^{-2\ln n}\right)=O(1/n).
\end{align*}
Thus, whp
$$
\mu(G)\leq n^{4^{2\ln n/(-\varepsilon-\ln(1-\varepsilon))}\cdot \ln^4 n}=2^{n^{2\ln 4/(-\varepsilon-\ln(1-\varepsilon))}\cdot\log_2 n\cdot \ln^4 n}<2^{n^{c_2}},
$$
where $c_2=\ln 17/(-\varepsilon-\ln(1-\varepsilon))$, say.

Now, let $np\geq 1-\varepsilon$ and $G\sim G(n,p)$. It is sufficient to show that there exists a strictly decreasing continuous function $c_1=c_1(\varepsilon)\in(0,1)$ such that whp $\mu(G)\geq 2^{n^{(1-o(1))c_1}}$ and $c_1(0+)=1$. Without loss of generality we may assume that $np=O(1)$. 
 Set 
\begin{itemize}
\item $c_1=1+3(\varepsilon+\ln(1-\varepsilon))=1-\Theta(\varepsilon^2)$ and $k=\left\lfloor\frac{1-c_1}{-\varepsilon-\ln(1-\varepsilon)}\ln n\right\rfloor=\lfloor 3\ln n\rfloor$, if $1-\varepsilon>\frac{1}{2}$;
\item $c_1=\frac{1}{2(1-\varepsilon-\ln(1-\varepsilon))}$ and $k=\left\lfloor\frac{1-c_1}{np-1-\ln(np)}\ln n\right\rfloor$, if $1-\varepsilon\leq np\leq \frac{1}{2}$.\footnote{Even though $c_1=c_1(\varepsilon)$ is not continuous at $1/2$, it can be decreased in the left neighbourhood of $1/2$ in order to make it continuous.}
\end{itemize}
  Let us prove that whp $G$ contains many non-isomorphic tree connected components on $k$ vertices.

\begin{claim}
Whp 
\begin{itemize}
\item $G$ contains at least $n^{(1-o(1))c_1}$ connected components isomorphic to a tree on $k$ vertices;
\item $G$ does not contain two connected components isomorphic to the same tree on $k$ vertices.
\end{itemize}
\label{cl:tree_components}
\end{claim}

\begin{proof}
Let $X$ be the number of connected components of $G$ isomorphic to a tree on $k$ vertices. Then
\begin{align*}
 \mathbb{E}X={n\choose k} k^{k-2} p^{k-1}(1-p)^{k(n-k)+{k\choose 2}-(k-1)}&=
  (1+o(1)) \frac{n^k}{k!} k^{k-2} p^{k-1}e^{-pkn}\\
 &=\frac{1+o(1)}{\sqrt{2\pi k}k^2p}e^{(1+\ln(np)-np)k}.
\end{align*}
If  $1-\varepsilon>\frac{1}{2}$, then 
$$
\mathbb{E}X
=\Omega\left(\frac{n}{(\ln n)^{2.5}}e^{k(\varepsilon+\ln(1-\varepsilon))}\right)=\Omega\left(\frac{n^{c_1}}{(\ln n)^{2.5}}\right).
$$
If $1-\varepsilon\leq np\leq \frac{1}{2}$, then
\begin{equation}
\mathbb{E}X
=\Theta\left(\frac{n}{(\ln n)^{2.5}}n^{-1+c_1}\right)=\Theta\left(\frac{n^{c_1}}{(\ln n)^{2.5}}\right).
\label{eq:expect_X_big_eps}
\end{equation}
Moreover,
$$
 \mathbb{E}X(X-1)={n\choose k}{n-k\choose k} k^{2(k-2)} p^{2(k-1)}(1-p)^{2k(n-2k)+k^2+2({k\choose 2}-(k-1))}=(1+o(1))(\mathbb{E}X)^2.
$$
Therefore, by Chebyshev's inequality, $X/\mathbb{E}X\stackrel{\mathbb{P}}\to 1$.

Now, let $Y$ be the number of pairs $(T_1,T_2)$ of connected components of $G$ isomorphic to the same tree on $k$ vertices. We get
$$
 \mathbb{E}Y\leq \mathbb{E}X{n\choose k} k! p^{k-1}(1-p)^{k(n-2k)}  \leq(1+o(1)) \mathbb{E}X\cdot n (np)^{k-1}(1-p)^{kn}=
 O\left(n\cdot \mathbb{E}X\cdot e^{k(\ln(np)-np)}\right).
$$ 
If  $1-\varepsilon>\frac{1}{2}$, then 
$$
\mathbb{E}Y=O(n^2 e^{-k})=O(1/n),
$$
since $\ln(np)-np\leq -1$ and $\mathbb{E}X\leq n$. If $1-\varepsilon\leq np\leq \frac{1}{2}$, then, due to~\eqref{eq:expect_X_big_eps},
$$
\mathbb{E}Y=O\left(n^{1+c_1} e^{-(1-c_1)(1+1/(np-1-\ln(np)))\ln n}\right)=O\left(n^{2c_1-(1-c_1)/(np-1-\ln(np))}\right)=o(1),
$$
since 
$$
 c_1<\frac{1}{2(1-\varepsilon-\ln(1-\varepsilon))-1}\leq\frac{1}{2(np-\ln(np))-1}.
$$
In both cases, $\mathbb{E}Y=o(1)$ and, thus, whp there are no such pairs $(T_1,T_2)$.

\end{proof}




From Claim~\ref{cl:tree_components} it follows that, whp $G$ contains at least $n^{(1-o(1))c_1}$ connected components $T_1,\ldots,T_m$ such that, for every $i\neq j$, $T_i$ and $T_j$ are not isomorphic. Therefore, for any two sets of indices $\mathcal{I}_1,\mathcal{I}_2\subset[m]$, disjoint unions of trees $\sqcup_{i\in\mathcal{I}_1}T_i$ and $\sqcup_{i\in\mathcal{I}_2}T_i$ are not isomorphic. Thus, whp $G$ contains at least $2^m\geq 2^{n^{(1-o(1))c_1}}$ non-isomorphic induced subgraphs, completing the proof.

\section{Random regular graphs}
\label{sc:regular}

In this section, we prove Theorem~\ref{th:regular}, discuss its tightness, and then prove Theorem~\ref{th:expanders}.

\subsection{Proof of Theorem~\ref{th:regular}}
\label{sc:regular-proof}

Let $\varepsilon>0$ be small enough. Following the same lines as in the proof of the fourth part of Theorem~\ref{th:bin_dense} in Section~\ref{sc:proof_4}, we note that it is sufficient to prove the following analogues of Claim~\ref{cl:concentrated_subgraphs} and the bound~\eqref{eq:binomial_p4_UU'}.

  Let $m\in\mathcal{J}_n$, $U\in{[n]\choose m}$. We also assume that $d\gg\frac{1}{\varepsilon^4}$. 

\begin{claim}
Whp every subset $W\subset U$ of size at least $\varepsilon^2 n$ spans the number of edges satisfying 
$$
\left||E(G[W])|-\frac{d|W|^2}{2n}\right|\leq \frac{d|W|^2}{4n}.
$$
\label{cl:concentrated_subgraphs-regular}
\end{claim}

\begin{claim}
Let $\mathcal{U}(U)$ be the family of all $m$-subsets of $[n]$ that have less than $m-\varepsilon^2 n$ common vertices with $U$. Then whp there is no set $U'\in\mathcal{U}(U)$ such that $G[U]\cong G[U']$.
\label{cl:non-isomorphic-disjoint-regular}
\end{claim}

The rest of the proof of Theorem~\ref{th:regular} is devoted to the proofs of these two claims.

\begin{proof}[Proof of Claim~\ref{cl:concentrated_subgraphs-regular}]
Since the largest absolute value $\lambda(G)$ of a non-trivial eigenvalue of $G$ is less than $2\sqrt{d-1}+1$ whp~\cite{Bordenave,Friedman,Puder}, by applying the expander mixing lemma~\cite[Corollary 9.2.6]{Alon}, we get that whp, for every subset $W\subset U$ of size at least $\varepsilon^2 n$,
$$
\left||E(G[W])|-\frac{d|W|^2}{2n}\right|<\left(\sqrt{d-1}+1\right)|W|< \frac{d|W|^2}{4n},
$$
completing the proof.
\end{proof}

\begin{proof}[Proof of Claim~\ref{cl:non-isomorphic-disjoint-regular}]
We will use the following result, due to McKay~\cite{McKay}. Let $H=H(n)$ be an arbitrary graph with maximum degree at most $d$ on $[n]$, and let $|E(H)|=\omega(1)$. Then
\begin{equation}
\label{eq:McKay}
 \mathbb{P}(H\subset G)=(1+o(1))\frac{\prod_{j\in[n]}d(d-1)\ldots(d-\mathrm{deg}_H j+1)}{2^{|E(H)|}(dn/2)(dn/2-1)\ldots(dn/2-|E(H)|+1)}.
\end{equation}

For every $U'\in\mathcal{U}(U)$ such that $G[U']\cong G[U]$, there should exist an isomorphism $U\to U'$ that moves all edges from $E(G[U\setminus U'])$. If $U$ has subgraphs with the number of edges concentrated as in Claim~\ref{cl:concentrated_subgraphs-regular}, the number of moved edges is at least $\frac{1}{4}\varepsilon^4 dn$. Then, due to~\eqref{eq:McKay} and since $d\gg1/\varepsilon^4$,
\begin{align*}
 \mathbb{P}\biggl(\exists U'\in\mathcal{U}(U)\,\,\, G[U]\cong G[U']\biggr)
 &\leq o(1)+2^n n! \left(\frac{d}{n(1-\varepsilon^4)}\right)^{\frac{1}{4}\varepsilon^4 dn}=o(1),
\end{align*}
completing the proof.
\end{proof}
 
\subsection{Tightness}
\label{sc:regular-tight}

Theorem~\ref{th:regular} is tight in the sense that, for every $n$-vertex graph $G$ with a bounded maximum degree, $\mu(G)\leq 2^{(1-\Theta(1))n}$.

\begin{theorem}
\label{th:max_degree}
For every $C>0$ there exists $\varepsilon>0$ such that, for all large enough $n$, every graph $G$ with $n$ vertices and maximum degree at most $C$ has $\mu(G)\leq 2^{(1-\varepsilon)n}$.
\end{theorem}

\begin{proof}
Let $G$ be a graph on $[n]$ with maximum degree at most $C$. Let $\mathbf{U}$ be a uniformly random subset of $[n]$. Let $X$ be the number of sets $U'\neq\mathbf{U}$ such that $G[\mathbf{U}]\cong G[U']$. It suffices to prove that, for a sufficiently small $\varepsilon>0$, 
$$
   \mathbb{P}\biggl(X\geq 2^{\varepsilon n}\biggr)\geq 1-2^{-\varepsilon n}.
$$
Indeed, if this is the case, let us partition the set of all subsets of $[n]$ into equivalence classes, where two sets are equivalent whenever they induce isomorphic subgraphs of $G$.  Then $2^{[n]}=\mathcal{X}_1\sqcup\mathcal{X}_2$, where $\mathcal{X}_1$ is the union of all equivalence classes of size less than $2^{\varepsilon n}$ and it has cardinality $2^n\cdot\mathbb{P}(X<2^{\varepsilon n})<2^{n-\varepsilon n}$. Since every equivalence class that is a subset of $\mathcal{X}_2$ has size at least $2^{\varepsilon n}$, we get that the total number of classes is at most $2^{n-\varepsilon n+1}$.

  We then find a set $I\subset[n]$ of size at least $\frac{n}{C^2}$ such that 1-neighbourhoods in $G$ of all vertices from $I$ are disjoint. For every $v\in I$, let $\delta_v$ be the number of neighbours of $v$ in $\mathbf{U}$. Let 
$$
\xi_v=\1_{v\in \mathbf{U},\, \delta_v=0},\,\,\xi=\sum_{v\in I}\xi_v,\quad\text{ and }\quad
\eta_v=\1_{v\not\in \mathbf{U},\, \delta_v=0},\,\,\eta=\sum_{v\in I}\eta_v.
$$
Since all $\xi_v$, $v\in I$, are independent and stochastically dominate $Bernoulli(2^{-C-1})$, and the same applies to $\eta_v$, we get that
$$
 \max\left\{\mathbb{P}\left(\xi<\frac{n}{100 C^2 2^{C}}\right),
 \mathbb{P}\left(\eta<\frac{n}{100 C^2 2^{C}}\right)\right\}
 <2^{-\varepsilon n-1},
$$ 
for sufficiently small constant $\varepsilon=\varepsilon(C)>0$. In particular, we may assume that $\varepsilon<\frac{1}{100 C^2 2^{C}}$.  For every set $U'$ obtained from $\mathbf{U}$ by removing vertices $v\in\mathbf{U}$ with $\delta_v=0$ and by adding the same number of vertices $u\not\in\mathbf{U}$ with $\delta_u=0$, we have that $G[\mathbf{U}]\cong G[U']$. Indeed, all such vertices belong to $I$ and, therefore, they are not adjacent and so they form independent sets, both in $G[\mathbf{U}]$ and $G[U']$. Therefore, $X\geq 2^{\varepsilon n}$ with probability at least $1-2^{-\varepsilon n}$. 
\end{proof}

\begin{remark}
The bound on the maximum degree in Theorem~\ref{th:max_degree} cannot we weakened to a linear bound on the number of edges: there exists a sequence of graphs $G=G(n)$ with $V(G(n))=[n]$ with at most $2n$ edges so that $\mu(G)=2^{n-O(\sqrt{n})}$. Such $G$ can be constructed in the following way. Start from a path $P$ on $100\sqrt{n}$ vertices, and then draw two edges from every vertex outside of $P$ to $P$ so that (1) every two vertices on $P$ that have a neighbour outside $P$ are at distance at least 10 in $P$ (here, 10 is some rather arbitrarily chosen and not necessarily optimal constant for which the argument works); (2) every vertex on $P$ that has a neighbour outside $P$ is at distance at least 10 from both leaves in $P$; (3) any two vertices outside $P$ have different neighbourhoods in $P$. Then, for every $U\subset V(G)$ such that $V(P)\subset U$ there is at most one other $U'\subset V(G)$ such that $V(P)\subset U'$ and $G[U]\cong G[U']$. Indeed, assume $U\neq U'$ satisfy $V(P)\subset U,U'$ and $G[U]\cong G[U']$. Then any isomorphism $\varphi:U\to U'$ between these two graphs preserves the property of vertices to have degrees of all neighbours equal 2. Moreover, $\varphi$ preserves the property of having degree more than 2. This immediately implies that $\varphi(V(P))=V(P)$, and so $\varphi|_{V(P)}$ is an automorphism of $P$. There are exactly two such automorphisms, and it is easy to see that both automorphisms admit at most one extension to the entire $U$: the trivial automorphism extends only to the trivial automorphism of $G[U]$, and the non-trivial automorphism extends to an isomorphism between $G[U]$ and $G[U']$ for at most one~$U'$. We also note that this construction can be generalised to get, for every $\varepsilon>0$, graphs $G$ with $O_{\varepsilon}(n)$ edges and with $\log_2\mu(G)\geq n-n^{\varepsilon}$. This can be achieved by shrinking $P$ to a path of size $n^{1/k}$, where $k$ is a positive integer so that $1/k<\varepsilon$, and by attaching every vertex outside of $P$ to $k$ vertices on $P$ in a similar manner as above. Nevertheless, such a graph has around $kn$ edges, which grows with $k$. It would be interesting to know the maximum possible $\mu(G)$ achieved by graphs $G$ on $n$ vertices with a given number of edges $m=|E(G)|$.
\end{remark}

\begin{remark}
Using similar ideas, it is easy to get a generalisation of Theorem~\ref{th:regular} for all $\omega(1)\leq d\leq n/2$. For such $d$, whp $G\sim G_{n,d}$ has $\mu(G)=2^{(1-o(1))n}$. 
\end{remark}

\subsection{Exponentially many subgraphs for all $d$: proof of Theorem~\ref{th:expanders}}
\label{sc:regular-proof2}

Let $d\geq 3$ be a constant and let $G\sim G(n,d)$. Due to~\cite[Theorem 1]{FJ} (see also~\cite{EFMN}) there exists a constant $c=c(d)>0$ such that, whp $G$ contains an induced path of length at least $c(d)n$. Let $\varepsilon\in(0,c(d))$ be small enough as a function of $d$, and let $P=(v_1\ldots v_{\ell+1})\subset G$ be an induced path of length exactly $\ell:=\lceil\varepsilon n\rceil$. For every $v_i\in V(P)$, consider an arbitrary edge $\{v_i,u_i\}$ that does not belong to $P$. Since the path $P$ is induced, $u_i\notin V(P)$. Let $U=\{u_1,\ldots,u_{\ell+1}\}$. Note that some of these vertices may coincide, in which case $|U|<\ell+1$. Nevertheless, let us show that most vertices in $\{u_1,\ldots,u_{\ell+1}\}$ do not coincide and have degree 1 in $G[V(P)\cup U]$ whp (independently of the choice of all $u_i$). This would follow from the fact that, 
for a small constant $\varepsilon'\gg\varepsilon$ (say $\varepsilon'=\sqrt{1/\ln(1/\varepsilon)}$ is enough for our goals), the number of edges induced by $U\cup V(P)$ is at most $(1+\varepsilon'/2)|U\cup V(P)|\leq |U\cup V(P)|+\varepsilon'\ell$. Indeed, due to~\eqref{eq:McKay}, whp $G$ does not have subgraphs $H$ with $\varepsilon n\leq |V(H)|\leq 2\varepsilon n$ and $|E(H)|\geq(1+\varepsilon'/2)|V(H)|$ --- the expected number of such subgraphs is at most
\begin{align*}
 \sum_{v=\varepsilon n}^{2\varepsilon n}{n\choose v}{{v\choose 2}\choose (1+\varepsilon'/2)v}\left(\frac{2d}{n}\right)^{(1+\varepsilon'/2)v}&\leq
 \sum_{v=\varepsilon n}^{2\varepsilon n}\exp\left[v\ln\frac{en}{v}+(1+\varepsilon'/2)v\left(\ln(ev)-\ln\frac{n}{2d}\right) \right]\\
 &=\sum_{v=\varepsilon n}^{2\varepsilon n}\exp\left[v\ln\frac{en(ev)^{1+\varepsilon'/2}}{v(n/2d)^{1+\varepsilon'/2}} \right]\\
&\leq\sum_{v=\varepsilon n}^{2\varepsilon n}\exp\left[v\ln((de)^3(v/n)^{\varepsilon'/2}) \right]\\
&\leq\sum_{v=\varepsilon n}^{2\varepsilon n}\exp\left[v\left(3\ln(de)-\frac{\varepsilon'}{2}\cdot\ln\frac{1}{2\varepsilon}\right) \right]\\
&\leq\sum_{v=\varepsilon n}^{2\varepsilon n}\exp\left[-\frac{1}{3}v\left(\ln\frac{1}{\varepsilon}\right)^{1/2} \right]=\exp(-\Theta(n)).
\end{align*}
Then, whp $G[U\cup V(P)]$ has excess at most $\varepsilon'\ell$ and, therefore, there exists a set $U^*\subset U$ of size at least $(1-2\varepsilon')\ell$ such that 
\begin{itemize}
\item vertices $v_1,v_2,v_{\ell},v_{\ell+1}$ do not have neighbours in $U^*$;
\item each vertex from $U^*$ contains exactly one neighbour in $V(P)\cup U$, and this neighbour lies on $P$. 
\end{itemize}
Note that the second condition immediately implies that there are no two vertices in $U^*$ that share a neighbour on $P$. Indeed, otherwise, let $u_i,u_j\in U^*$ be both adjacent to $v_i$, say. Then $u_j$ has two neighbours on $P$, $v_i$ and $v_j$ --- a contradiction.

Let us show that the {\it comb} $H:=G[V(P)\cup U^*]$ has exponentially (in $n$) many non-isomorphic induced subgraphs. Assume that there are two different subsets $W,W'\subset U^*$ such that $G[V(P)\cup W]\cong G[V(P)\cup W']$. Let $\varphi:V(P)\cup W\to V(P)\cup W'$ be an isomorphism between these two graphs. Since all leaves of the tree $G[V(P)\cup W]$ that belong to $W$ adjacent to vertices of degree 3 (in contrast to the leaves $v_1,v_{\ell+1}$), $\varphi$ has to map $W$ to $W'$. Therefore, $\varphi|_{V(P)}$ is an automorphism of $P$. If $\varphi|_{V(P)}$ is identity, then $W=W'$. Otherwise, when $\varphi|_{V(P)}$ is a non-trivial involution, there are at most two ways to choose $W'$, for a given $W$, so that $G[V(P)\cup W]\cong G[V(P)\cup W']$. 
 We conclude that the number of non-isomorphic induced subgraphs $F\subset H$ with $V(F)\supseteq V(P)$ is at least $2^{|U^*|-1}\geq 2^{(1-2\varepsilon')\ell-1}\geq2^{\varepsilon(1-2\varepsilon')n-1}=2^{\Theta(n)}$, completing the proof of the theorem.

\appendix

\section{Proof of the lower bound in Claim~\ref{cl}}
\label{sc:appendix}

Here, we prove

\begin{claim}
Whp $\xi\geq \alpha_n+(1-o(1))\beta_n$, where $\alpha_n$ and $\beta_n$ are defined in~\eqref{eq:alpha_beta}.
\label{cl:cl_lower}
\end{claim} 

\begin{proof}
Letting 
$$
\mathcal{B}_{x,x'}(k)=\left\{\xi_{x,x'}-nq>k\right\}\quad\text{ and }\quad k=\sqrt{(4-\varepsilon)nq(1-q)\ln n},
$$
we get
$$
 \mathbb{P}\left(\mathcal{B}_{x,x'}(k)\right)=\frac{n^{-2+\varepsilon+o(1)}}{\sqrt{2(4+\varepsilon)\pi\ln n}},\quad \mathbb{P}\left(\xi_{x,x'}-nq=k+O(1)\right)=O\left(\frac{1}{\sqrt{nq}}\cdot\mathbb{P}\left(\mathcal{B}_{x,x'}(k)\right)\right).
$$
Let $X$ be the number of pairs $\{x,x'\}$ such that the event 
$\mathcal{B}_{x,x'}(k)$ holds. We get that 
\begin{equation}
\mathbb{E}[X]=n^{\varepsilon-o(1)}.
\label{eq:expectation_large}
\end{equation}
One can expect that $\mathrm{Var}[X]=o((\mathbb{E}[X])^2)$ implying the desired assertion due to Chebyshev's inequality. However, the rare event that there exist a vertex with a large degree contributes superfluously to the variance (cf.~\cite{Rodi}). Thus, we consider the following ``pruned'' version of the random variable $X$: 
$$
 \tilde X=\sum_{x,x'}\1_{\mathcal{\tilde B}_{x,x'}(k)},\quad\text{ where}
$$
$$
\mathcal{\tilde B}_{x,x'}(k):=\mathcal{B}_{x,x'}(k)\cap\left\{d^-\leq\mathrm{deg}(x)\leq d^+\right\}\cap\left\{d^-\leq\mathrm{deg}(x')\leq d^+\right\},
$$
$$
 d^-=np-\sqrt{2np(1-p)\ln n},\quad d^+=np+\sqrt{2np(1-p)\ln n}.
$$ 
For fixed $x,x'$, the event 
$\mathcal{B}_{x,x'}(k)\cap\left\{\mathrm{deg}(x)> d^+\right\}
$  implies 
\begin{align*}
\tilde\xi_{x,x'} & :=2|N(x)\cap N(x')|+|N(x)\setminus N(x')|+|[n]\setminus(N(x)\cup N(x'))|\\
&>n(1-p+2p^2)+\sqrt{2p(1-p)n\ln n}\left(1+\sqrt{(4-\varepsilon)q}\right).
\end{align*}
The latter random variables is a sum of $n$ independent random variables $\xi_1,\ldots,\xi_n$, where 
$$
\mathbb{P}(\xi_i=2)=p^2,\,\,\mathbb{P}(\xi_i=1)=1-p,\,\,\text{ and }\,\,\mathbb{P}(\xi_i=0)=p-p^2.
$$
Thus, $\mathbb{E}[\tilde\xi_{x,x'}]=n(1-p+2p^2)$ and $\mathrm{Var}[\tilde\xi_{x,x'}]=np(1-p)(1+4p^2)$. By the local limit theorem,
\begin{multline*}
 \mathbb{P}\left(\tilde\xi_{x,x'}-\mathbb{E}[\tilde\xi_{x,x'}]>\sqrt{2p(1-p)n\ln n}\left(1+\sqrt{(4-\varepsilon)q}\right)\right) \\
 =\frac{1+o(1)}{\sqrt{2\pi}}\int_{\left(1+\sqrt{(4-\varepsilon)q}\right)\sqrt{\frac{2\ln n}{1-p+2p^2}}}^{\infty}e^{-t^2/2}dt
= n^{-(2+o(1))\frac{(1+\sqrt{(4-\varepsilon)q)^2}}{1-p+2p^2}}=o(n^{-2})
\end{multline*}
since $1-p+2p^2\leq 1$ while $1+\sqrt{(4-\varepsilon)q)^2}>1$. In a similar way,
$$
\mathbb{P}\left(\mathcal{B}_{x,x'}(k)\cap\left\{\mathrm{deg}(x)< d^-\right\}\right)=o(n^{-2}).
$$
Thus,
\begin{equation}
 \mathbb{E}[\tilde X]\geq \mathbb{E}[X]-2n^2\mathbb{E}[\tilde\xi_{x,x'}]=\mathbb{E}[X]-o(1).
\label{eq:expectation_tilde_large}
\end{equation}

Consider two disjoint pairs of vertices $\{x_1,x_1'\}$ and $\{x_2,x_2'\}$. Let $\eta_{x_i,x'_i}$, $i\in\{1,2\}$, be the number of vertices in $[n]\setminus\{x_1,x'_1,x_2,x'_2\}$ that are either adjacent to both $x_i,x'_i$, or non-adjacent to both; and let $\eta_{x_i}$ (or $\eta_{x'_i}$) be the number of neighbours of $x_i$ (or $x'_i$) in $[n]\setminus\{x_1,x'_1,x_2,x'_2\}$. Then
\begin{align}
 \mathbb{P}\left(\bigcap_{i=1,2}\mathcal{\tilde B}_{x_i,x_i'}(k)\right) &\leq\mathbb{P}\left(\bigcap_{i=1,2}\left\{\eta_{x_1,x_1'}>nq+k-2 , \, \eta_{x_i}\leq d^+, \, \eta_{x'_i}\leq d^+\right\}\right)\notag\\
 &=
 \prod_{i=1,2}\mathbb{P}\left(\eta_{x_1,x_1'}>nq+k-2,\,\eta_{x_i}\in[d^-, d^+],\,\eta_{x'_i}\in[d^-, d^+]\right)\notag\\
 &\leq
\prod_{i=1,2}\mathbb{P}\left(\mathcal{B}_{x_1,x_1'}(k-2)\cap\left\{\mathrm{deg}(x_i),\mathrm{deg}(x'_i)\in[d^-+3, d^++3]\right\}\right)\notag\\
 &=\left(1+O\left(\frac{1}{\sqrt{np}}\right)\right)\mathbb{P}\left(\mathcal{\tilde B}_{x_1,x_1'}(k)\right)
 \mathbb{P}\left(\mathcal{\tilde B}_{x_2,x_2'}(k)\right).
 \label{eq:variance_disjoint}
\end{align}

Now, let us consider pairs $\{x,x'\}$ and $\{x',x''\}$. As above, $\eta_{x,x'}$ and $\eta_{x',x''}$ are numbers of vertices in $[n]\setminus\{x,x',x''\}$ that are either adjacent to both $x,x'$ (both $x',x''$), or non-adjacent to both; for $y\in\{x,x',x''\}$, $\eta_{y}$ is the number of neighbours of $y$ in $[n]\setminus\{x,x',x''\}$. Let 
$$
\mathcal{M}=\mathbb{Z}\cap\left[d^-,d^+\right],\quad
\mathcal{M}_{\varepsilon}=\mathbb{Z}\cap\left[d^-(\varepsilon),d^+(\varepsilon)\right],\quad\text{ where}
$$
$$
 d^-(\varepsilon)=np-\sqrt{(2-\varepsilon)np(1-p)\ln n},\quad
 d^+(\varepsilon)=np+\sqrt{(2-\varepsilon)np(1-p)\ln n}.
$$
For $m\in\mathcal{M}$, we consider independent $\zeta_m\sim\mathrm{Bin}(m,p)$ and $\zeta'_m\sim\mathrm{Bin}(n-m,1-p)$. Then,
\begin{align*}
 \mathbb{P}\left(\mathcal{\tilde B}_{x,x'}(k)\cap \mathcal{\tilde B}_{x',x''}(k)\right) &\leq\mathbb{P}\left(\eta_{x,x'}>nq+k-2;\,\,\eta_{x',x''}>nq+k-2;\,\,\eta_x,\eta_{x'},\eta_{x''}\in[d^-, d^+]\right)\\
 &\leq
 \sum_{m\in\mathcal{M}}\mathbb{P}(\eta_{x'}=m)\left[\mathbb{P}\left(\zeta_m+\zeta'_m>nq+k-2\right)\right]^2
\end{align*}
and also
\begin{align*}
 \mathbb{P}\left(\mathcal{\tilde B}_{x,x'}(k)\cap \mathcal{\tilde B}_{x',x''}(k)\right) &\leq\mathbb{P}\left(\mathcal{\tilde B}_{x,x'}(k)\right)\max_{m\in\mathcal{M}_{\varepsilon}}\mathbb{P}\left(\eta_{x',x''}>nq+k-2\mid \eta_{x'}=m\right)\\
 &\quad \quad +\sum_{m\in\mathcal{M}\setminus\mathcal{M}_{\varepsilon}}\mathbb{P}(\eta_{x'}=m)\left[\mathbb{P}\left(\zeta_m+\zeta'_m>nq+k-2\right)\right]^2.\\
 &\leq
 \mathbb{P}\left(\mathcal{\tilde B}_{x,x'}(k)\right) \max_{m\in\mathcal{M}_{\varepsilon}}\mathbb{P}\left(\zeta_m+\zeta'_m>nq+k-2\right)\\
  &\quad \quad +\sum_{m\in\mathcal{M}\setminus\mathcal{M}_{\varepsilon}}\mathbb{P}(\eta_{x'}=m)\left[\mathbb{P}\left(\zeta_m+\zeta'_m>nq+k-2\right)\right]^2.
\end{align*}
If $p\leq n^{-2/3}$, then any $m$ from $\mathcal{M}$ equals $np(1+o(1))$. Thus, $$
\mathbb{P}(\zeta_m\geq 10)\leq (np(1+o(1)))^{10}p^{10}\leq ((1+o(1))n^{-1/3})^{10}=o(n^{-3}).
$$
In this case,
\begin{align*}
\mathbb{P}\left(\mathcal{\tilde B}_{x,x'}(k)\cap \mathcal{\tilde B}_{x',x''}(k)\right)
&\leq \mathbb{P}\left(\mathcal{\tilde B}_{x,x'}(k)\right)\mathbb{P}\left(\mathrm{Bin}(n-d^-(\varepsilon),1-p)>nq+k-11\right)+o(n^{-3})\\
&\quad +\sum_{m\in\mathcal{M}\setminus\mathcal{M}_{\varepsilon}}\mathbb{P}(\eta_{x'}=m)\left[\mathbb{P}\left(\mathrm{Bin}(n-d^-,1-p)>nq+k-11\right)\right]^2\\
&\leq \mathbb{P}\left(\mathcal{\tilde B}_{x,x'}(k)\right)\mathbb{P}\left(n-d^-(\varepsilon)-\mathrm{Bin}(n,p)>nq+k-20\right)\\
&\quad +\mathbb{P}(\eta_{x'}\in\mathcal{M}_{\varepsilon})\left[\mathbb{P}\left(n-d^--\mathrm{Bin}(n,p)>nq+k-20\right)\right]^2+o(n^{-3})\\
&\leq \mathbb{P}\left(\mathrm{Bin}(n,p)<np-\left(\sqrt{8-2\varepsilon}-\sqrt{2-\varepsilon}\right)\sqrt{np(1-p)\ln n}+21\right)\\
&\quad \times \mathbb{P}\left(\mathcal{\tilde B}_{x,x'}(k)\right)+o(n^{-3}) +\sqrt{np\ln n}\cdot\mathbb{P}(\eta_{x'}=d^-(\varepsilon))\\
&\quad\quad\times\mathbb{P}\left[\mathrm{Bin}(n,p)<np-\left(\sqrt{8-2\varepsilon}-\sqrt{2}\right)\sqrt{np(1-p)\ln n}+21\right]^2.
\end{align*}

Since $\sqrt{8-2\varepsilon}-\sqrt{2-\varepsilon}>\sqrt{2+\varepsilon^2/16}$, by the de Moivre--Laplace limit theorem, we get
\begin{align}
 \mathbb{P}\left(\mathcal{\tilde B}_{x,x'}(k)\cap \mathcal{\tilde B}_{x',x''}(k)\right) &\leq
 \mathbb{P}\left(\mathcal{\tilde B}_{x,x'}(k)\right)o(1/n)+\sqrt{np\ln n}\cdot \frac{n^{-1+\varepsilon/2+o(1)}}{\sqrt{np}}\cdot \left(n^{-1+\varepsilon/2}\right)^2+o(n^{-3})\notag\\
 &=\mathbb{P}\left(\mathcal{\tilde B}_{x,x'}(k)\right)o(1/n)+n^{-3+3\varepsilon/2+o(1)}.
\label{eq:variance_depend_small}
\end{align}

Let $n^{-2/3}<p\leq n^{-1/2}\ln^2 n$. For $m\in\mathcal{M}$, $\mathbb{P}(\zeta_m\geq\ln^5 n-3)=o(1/n)$. As above, we get
\begin{align*}
\mathbb{P}\left(\mathcal{\tilde B}_{x,x'}(k)\cap \mathcal{\tilde B}_{x',x''}(k)\right)
&\leq \mathbb{P}\left(\mathrm{Bin}(n,p)<np-\left(\sqrt{8-2\varepsilon}-\sqrt{2-\varepsilon}\right)\sqrt{np(1-p)\ln n}+3\ln^5 n\right)\\
&\quad \times\mathbb{P}\left(\mathcal{\tilde B}_{x,x'}(k)\right) +o(n^{-3})+\sqrt{np\ln n}\cdot\mathbb{P}(\eta_{x'}=d^-(\varepsilon))\\
&\quad\quad\times\mathbb{P}\left[\mathrm{Bin}(n,p)<np-\left(\sqrt{8-2\varepsilon}-\sqrt{2}\right)\sqrt{np(1-p)\ln n}+3\ln^5 n\right]^2.
\end{align*}
Again, due to the de Moivre--Laplace limit theorem, we get
\begin{equation}
\mathbb{P}\left(\mathcal{\tilde B}_{x,x'}(k)\cap \mathcal{\tilde B}_{x',x''}(k)\right) \leq\mathbb{P}\left(\mathcal{\tilde B}_{x,x'}(k)\right)o(1/n)+n^{-3+3\varepsilon/2+o(1)}.
\label{eq:variance_depend_medium}
\end{equation}
In both cases,
$$
 \frac{\mathrm{Var}[\tilde X]}{(\mathbb{E}[\tilde X])^2}=
 \frac{\mathbb{E}[\tilde X(\tilde X -1)]-(\mathbb{E}[\tilde X])^2}{(\mathbb{E}[\tilde X])^2}+o(1)\leq O\left(\frac{1}{\sqrt{np}}\right)+\frac{o\left(\mathbb{E}[\tilde X]\right)+n^3\cdot n^{-3+3\varepsilon/2+o(1)}}{(\mathbb{E}[\tilde X])^2}=o(1),
$$
due to~\eqref{eq:expectation_large},~\eqref{eq:expectation_tilde_large},~\eqref{eq:variance_disjoint},~\eqref{eq:variance_depend_small},~\eqref{eq:variance_depend_medium}.

Finally, let $p>n^{-1/2}\ln^2 n$. Fix $m\in\mathcal{M}$. By the de Moivre--Laplace limit theorem, uniformly over all $\ell$ such that $|\ell-mp-(n-m)(1-p)|\leq\sqrt{np}\ln n$,
\begin{align*}
 \mathbb{P}\left(\zeta_m+\zeta'_m=\ell\right) &=\sum_{s=0}^m\mathbb{P}\left(\zeta_m=s)\mathbb{P}(\zeta'_m=\ell-s\right)\\
 &=(1+o(1))\sum_{s=mp-\sqrt{mp(\ln n)^{1.1}}}^{mp+\sqrt{mp(\ln n)^{1.1}}}\mathbb{P}\left(\zeta_m=s)\mathbb{P}(\zeta'_m=\ell-s\right)\\
 &=\frac{1+o(1)}{2\pi p(1-p)\sqrt{m(n-m)}}\sum_s\exp\left[-\frac{(s-mp)^2}{2mp(1-p)}-\frac{(n-m-\ell+s-(n-m)p)^2}{2(n-m)p(1-p)}\right]\\
 &=\frac{1+o(1)}{2\pi p(1-p)\sqrt{m(n-m)}} e^{-\frac{(\ell-n(1-p)-m(2p-1))^2}{2np(1-p)}}\int_{\mathbb{R}}e^{-\frac{(t-m((n-m)(2p-1)+\ell)/n)^2}{2m(n-m)p(1-p)/n}}dt\\
 &=\frac{1+o(1)}{\sqrt{2\pi p(1-p)n}} e^{-\frac{(\ell-n(1-p)-m(2p-1))^2}{2np(1-p)}}.
\end{align*}
Therefore,
\begin{align*}
 \mathbb{P}\left(\mathcal{\tilde B}_{x,x'}(k)\cap \mathcal{\tilde B}_{x',x''}(k)\right) &\leq
 \sum_{m\in\mathcal{M}}\mathbb{P}(\eta_{x'}=m)\left[\mathbb{P}\left(\zeta_m+\zeta'_m>nq+k-2\right)\right]^2\\
 &=
 \sum_{m\in\mathcal{M}}\frac{1+o(1)}{\sqrt{2\pi np(1-p)}}e^{-\frac{(np-m)^2}{2np(1-p)}}\left[\int_{\frac{(1-2p)(m-np)+k}{\sqrt{np(1-p)}}}^{\infty}\frac{1}{\sqrt{2\pi}}e^{-t^2/2}dt\right]^2\\
  &=
 \sum_{m\in\mathcal{M}}\frac{1+o(1)}{2\pi((1-2p)(m-np)+k)}e^{-\frac{(np-m)^2}{2np(1-p)}-\frac{((1-2p)(m-np)+k)^2}{np(1-p)}}\\
   &=
 \sum_{m\in\mathcal{M}}\frac{1+o(1)}{2\pi((1-2p)(m-np)+k)}e^{-\frac{\left(1+2(1-2p)^2\right)\left(m-np+\frac{2k(1-2p)}{1+2(1-2p)^2}\right)^2+\frac{2k^2}{1+2(1-2p)^2}}{2np(1-p)}}\\
 &\leq\frac{1+o(1)}{\sqrt{2\pi\ln n}\left(\sqrt{(4-\varepsilon)q}-(1-2p)\right)}e^{-\frac{k^2}{np(1-p)(1+2(1-2p)^2)}}\\
 &\quad\quad\quad\quad\quad\quad\quad
 \times\int_{\left(\frac{2k(1-2p)}{1+2(1-2p)^2}-\sqrt{2np(1-p)\ln n}\right)/\sqrt{\frac{np(1-p)}{1+2(1-2p)^2}}}^{\infty}\frac{1}{\sqrt{2\pi}}e^{-t^2/2}dt\\
 &=\Theta\left(\frac{1}{\ln n}\right)n^{-\frac{2(4-\varepsilon)q+\left(2\sqrt{(4-\varepsilon)q}(1-2p)-(1+2(1-2p)^2)\right)^2}{1+2(1-2p)^2}}.
\end{align*}
The function
\begin{multline*}
 \frac{2(4-\varepsilon)q+\left(2\sqrt{(4-\varepsilon)q}(1-2p)-(1+2(1-2p)^2)\right)^2}{1+2(1-2p)^2}\\
  =1+2\left(1-2p-\sqrt{(4-\varepsilon)(1-2p+2p^2)}\right)^2=:g(p)
\end{multline*}
is decreasing in $p$ since 
$$
\frac{d}{dp}\sqrt{(g(p)-1)/2}=\frac{2(1-2p)\sqrt{4-\varepsilon}-4\sqrt{1-2p+2p^2}}{2\sqrt{1-2p+2p^2}}<0.
$$
Therefore,
\begin{align}
 \mathbb{P}\left(\mathcal{\tilde B}_{x,x'}(k)\cap \mathcal{\tilde B}_{x',x''}(k)\right)=O\left(\frac{1}{\ln n}\right)n^{-\left(1+2\left(1-\sqrt{4-\varepsilon}\right)^2\right)}=n^{-11+\varepsilon+8\sqrt{1-\varepsilon/4}+o(1)}< n^{-3}.
 \label{eq:variance_depend_large}
 \end{align}
Due to~\eqref{eq:expectation_large},~\eqref{eq:expectation_tilde_large},~\eqref{eq:variance_disjoint},~\eqref{eq:variance_depend_large}, 
$$
 \frac{\mathrm{Var}[\tilde X]}{(\mathbb{E}[\tilde X])^2}=
 \frac{\mathbb{E}[\tilde X(\tilde X -1)]-(\mathbb{E}[\tilde X])^2}{(\mathbb{E}[\tilde X])^2}+o(1)\leq O\left(\frac{1}{\sqrt{np}}\right)+\frac{n^3\cdot n^{-3}}{(\mathbb{E}[\tilde X])^2}=o(1),
$$
completing the proof.
\end{proof}

\section{Proof of Claim~\ref{cl:boring}}
\label{app:cl_boring_proof}

First, let us notice that the inequality stated in the claim is true for all $p\in(1/3,1/2]$. Indeed, in this range, 
$$
 1+(1-p)^{5.4}< 1+(2/3)^{5.4}<1.2< 2^{1/2}\leq 2^{1-2p(1-p)}.
$$
Now, let $p\in(0,1/3]$. We get $(1-p)^{5.4}<e^{-5.4p}$ and
\begin{align*}
\frac{d}{dp}\left(1+e^{-5.4p}- 2\cdot e^{(-2p+2p^2)\ln 2}\right)
&=4\ln 2(1-2p) e^{(-2p+2p^2)\ln 2}-5.4 e^{-5.4p}\\
&<5.4e^{-2p\ln 2}\left(\frac{4\ln 2}{5.4}(1-2p) -e^{-(5.4-2\ln 2)p}\right).
\end{align*}
Since
$$
\frac{d}{dp}\left(\frac{4\ln 2}{5.4}(1-2p)-e^{-(5.4-2\ln 2)p}\right)=
-\frac{8\ln 2}{5.4}+(5.4-2\ln 2)e^{-(5.4-2\ln 2)p}>0,
$$
we get that 
$$
\frac{4\ln 2}{5.4}(1-2p)-e^{-(5.4-2\ln 2)p}\leq
\frac{4\ln 2}{3\cdot 5.4}-e^{-(5.4-2\ln 2)/3}<0.
$$
Thus, $1+e^{-5.4p}- 2\cdot e^{(-2p+2p^2)\ln 2}$ decreases and is strictly smaller than its value at $p=0$, implying that 
$$
 \gamma(p):=1+e^{-5.4p}-2\cdot e^{(-2p+2p^2)\ln 2}<\gamma(0)=0
$$ 
for all $p\in(0,1/3]$, completing the proof.

\end{document}